\documentclass[a4paper,11pt]{amsart}

\usepackage{ifpdf}
\usepackage{amsmath, amsthm, amsfonts, amssymb, amscd}
\usepackage[active]{srcltx}
\usepackage{color}
\usepackage[ansinew]{inputenc}

\newtheorem{theorem}{Theorem}[section] 
\newtheorem{corollary}[theorem]{Corollary}
\newtheorem{lemma}[theorem]{Lemma}

\newtheorem{proposition}[theorem]{Proposition}
\newtheorem*{proposition*}{Proposition}

\newtheorem*{question*}{Question}
\newtheorem*{theorem*}{Theorem}
\newtheorem*{claim*}{Claim}

\newtheorem*{corollary*}{Corollary}

\theoremstyle{definition}

\theoremstyle{remark}

\newtheorem*{remark*}{Remark}

\newcommand{\R}{\mathbb{R}}
\newcommand{\Z}{\mathbb{Z}}
\newcommand{\C}{\mathbb{C}}
\newcommand{\D}{\mathbb{D}}
\renewcommand{\S}{\mathbb{S}}

\def\eqalign#1{\null\,\vcenter{\openup\jot 
\ialign{\strut\hfil$\displaystyle{##}$&
$\displaystyle{{}##}$\hfil \crcr #1\crcr }}\,}

\def\eqalignno#1{\displ@y \tabskip=\@centering
\halign to\displaywidth{\hfil$\@lign\displaystyle{##}$
\tabskip=0pt &$\@lign\displaystyle{{}##}$

\hfil\tabskip=\@centering
$\llap{$\@lign##$}\tabskip=Opt\crcr #1\crcr}}

\usepackage{graphicx}

\begin{document}

\author[P. Le Calvez]{Patrice Le Calvez}
\address{Institut de Math\'ematiques de Jussieu-Paris Rive Gauche, IMJ-PRG, Sorbonne Universit\'e, Universit\'e Paris-Diderot, CNRS, F-75005, Paris, France \enskip \& \enskip Institut Universitaire de France}
\curraddr{}
\email{patrice.le-calvez@imj-prg.fr}

\title[Handel's fixed point theorem: a Morse theoretical point of view]{Handel's fixed point theorem: a Morse theoretical point of view}

\maketitle

\bigskip
\bigskip

\bigskip
\begin{abstract} 
    Michael Handel has proved in \cite{Ha} a fixed 
point theorem for an orientation preserving homeomorphism of the open 
unit disk, that turned out to be an efficient tool in the study of the dynamics of surface homeomorphisms.  The present article fits into a series of articles by the author \cite{LeC2} and by Juliana Xavier \cite{X1}, \cite{X2}, where proofs were given, related to the classical Brouwer Theory, instead of the Homotopical Brouwer Theory used in the original article. Like in \cite{LeC2}, \cite{X1} and \cite{X2}, we will use ``free brick decompositions'' but will present a more conceptual Morse theoretical argument. It is based on a new preliminary lemma, that gives a nice ``condition at infinity'' for our problem. \end{abstract}



\maketitle

\bigskip
\bigskip
\hfill {\it \`A la m\'emoire du professeur Yu Jia Rong}

\bigskip
\bigskip

\section{Introduction}

\bigskip
In this article, the unit circle  $\S=\{z\in\C\, \vert\, \vert z\vert =1\}$ of the complex plane is, as usual, oriented counterclockwise and this orientation induces an orientation on every interval of $\S$. If $a$ and $b$ are two distinct points of $\S$, we denote $(a,b)$ the open interval of $\S$ that joins $a$ to $b$ for the induced orientation. We will denote $\D=\{z\in\C\, \vert\, \vert z\vert =1\}$ the unit disk.  Let $f$ be a homeomorphism of $\D$ and $\Gamma\subset \D$ a simple loop that does not contain any fixed point of $f$.  Recall that the {\it index} $i(f,\Gamma)$ is the degree of the map 
$$s\mapsto {f(\Gamma(s))-\Gamma(s)\over \vert f(\Gamma(s))-\Gamma(s)\vert},$$ 
where $s\mapsto\Gamma(s)$ is a parametrization of $\Gamma$ defined on $\S$.
It is well known that if $i(f,\Gamma)\not=0$, the simply connected component of $\D\setminus \Gamma$ contains at least one fixed point of $f$. 

Let us state now the main result of the article:

\bigskip
\begin{theorem}\label{th:principal} Let $f$ be an orientation preserving homeomorphism of 
$\D$ and $p\geq 3$ an integer. We suppose that
 there exists a family $(z_i)_{i\in\Z/p\Z}$
in $\D$ and two families $(\alpha_i)_{i\in\Z/p\Z}$ and $(\omega_i)_{i\in\Z/p\Z}$
in $\S$, such that:

\begin{itemize}

\item  the $2n$ points $\alpha_i$, $\omega_i$, $i\in\Z/p\Z$, are distinct;

\item among these points, the unique one that belongs to $(\omega_{i-1},\omega_i)$
is
$\alpha_{i+1}$;

 \item for every $i\in\ \Z/p\Z$, it holds that
 $$\lim_{k\to-\infty} f^k(z_{i})=\alpha_{i},\enskip \lim_{k\to+\infty}
f^k(z_{i})=\omega_{i};$$

\item $f$ can be extended continuously to
$\D\cup\left(\bigcup_{i\in \Z/p\Z} \{\alpha_i,\omega_i\}\right)$.
\end{itemize}
Then, there exists a simple loop $\Gamma\subset \D$ that does not contain any fixed point of $f$,  such that $i(f,\Gamma)=1$. Consequently, $f$ has at least one fixed point.

\bigskip
\bigskip

\centerline{{\includegraphics[scale=0.5]{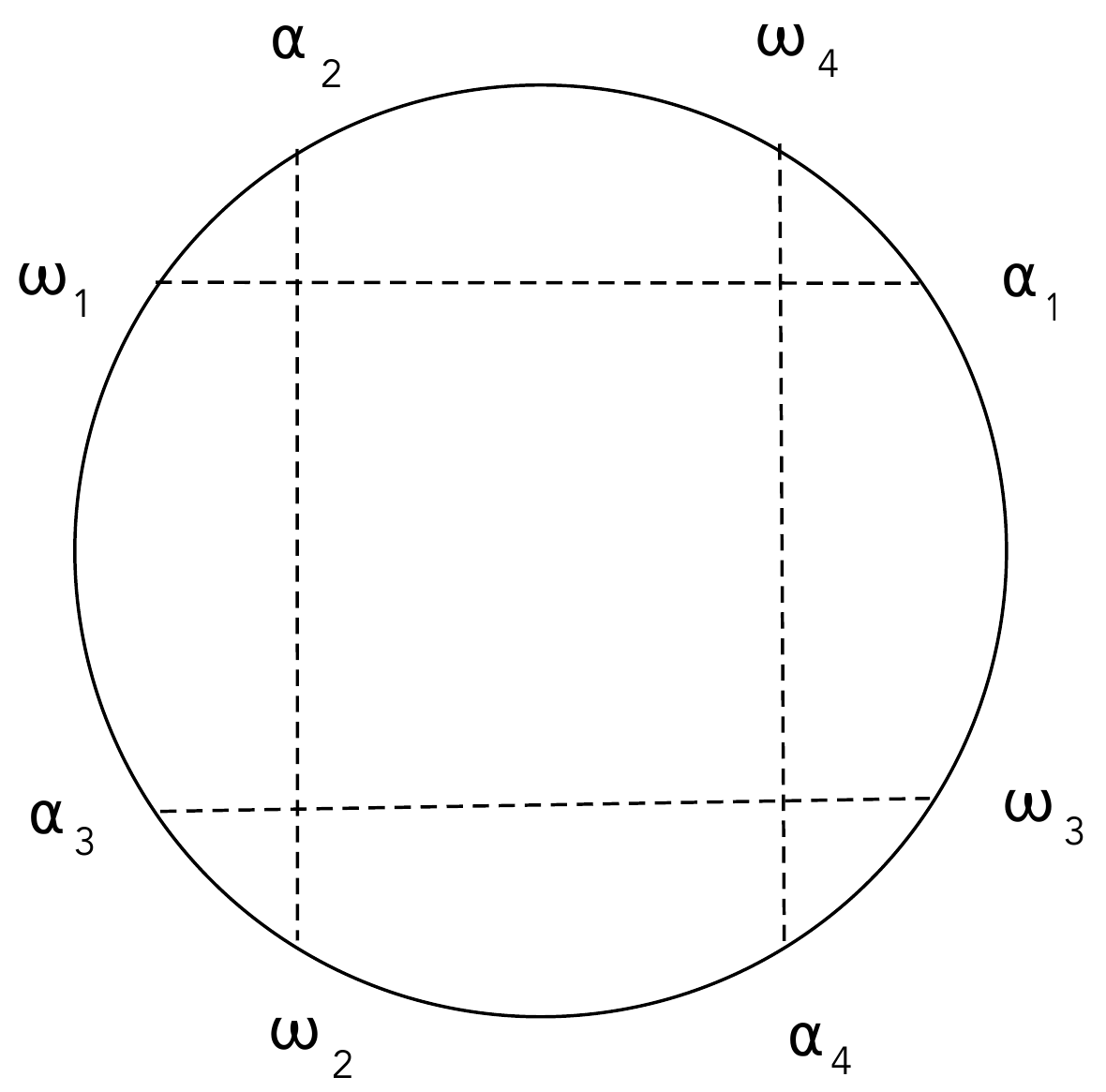}}}

\bigskip

\end{theorem}

This theorem is usually called {\it Handel's fixed point theorem}. The original version (see \cite{Ha}) is a little bit weaker: $f$ is supposed to extend continuously to $\S$ and the conclusion is that $f$ has at least one fixed point.  It is a highly important result in the study of the dynamics of surface homeomorphisms. For very nice applications, see Franks \cite{Fr2}, \cite{Fr3} or Matsumoto \cite{M}, for instance. The proof of Handel, which is judged very tough, is based on {\it Homotopical Brouwer Theory} which is the study of the mapping class group of the complement of finitely many orbits of a {\it Brouwer homeomorphism},  meaning a fixed point free orientation preserving homeomorphism of $\D$ (see \cite{Ha}, \cite{LeR2} or \cite{Ba}).

A proof of Theorem \ref{th:principal} was given in \cite{LeC2}, based on the classical { \it Brouwer Theory},  that used only topological arguments (and no geometrical argument). The fundamental fact is that the hypothesis of the theorem implies that $f$ is {\it recurrent}, a property known to imply the conclusion. This property will be defined very soon in the text, as all the notions we will introduce in the following lines. The proof given in \cite{LeC2} is a proof by contradiction. Assuming that the map is not recurrent, a preliminary lemma is proved in terms of {\it translation arcs}, giving some description of the dynamics in neighborhoods of the $\alpha_i$ and the $\omega_i$. A {\it maximal free brick decomposition} can be constructed on the complement of the fixed point set, that ``contains'' the dynamical elements described in the lemma. Then it is possible to show that this brick decomposition contains a {\it closed chain}, property that implies that the map is recurrent, which contradicts the hypothesis. A simplification of the proof was obtained by Xavier in \cite{X1} and \cite{X2}. Starting with the same preliminary lemma, a different argument was used to prove the existence of a closed chain in the brick decomposition, whose main interest is that it works for a wider class of brick decompositions that the ones considered in  \cite{LeC2}. It must be noticed that generalizations of Handel's fixed point theorem have been stated by Xavier, replacing the ``elliptic'' situation expressed in Theorem \ref{th:principal}  by  ``hyperbolic'' or ``degenerate'' situations.
 
In the present article, we display a new proof of Theorem \ref{th:principal}  which follows the same outline as the proofs of  \cite{LeC2},  \cite{X1} and  \cite{X2}: first an introductory lemma, like in the cited articles, then a construction of a brick decomposition similar to the construction done by Xavier and then a proof of the existence of a closed chain. The introductory lemma used in the present article, slightly different from the one appearing in \cite{LeC2},  \cite{X1} and  \cite{X2} is much easier to prove. It seems weaker that the old one but curiously turns out to be more efficient for our purpose. The arguments, relative to a brick decomposition, that appear later in the proof are more conceptual that those described in \cite{LeC2},  \cite{X1} and  \cite{X2}: a Morse theoretical reasoning using a condition at infinity given by the introductory lemma. It must be pointed out that one of the nicest applications of Handel's fixed point theorem is the extension by Matsumoto of the Arnold conjecture to Hamiltonian surface homeomorphisms and that the classical version for diffeomorphisms (see \cite{Fl} or  \cite{Si}) is related to Morse theory and its extensions, like Floer homology, as so many results of symplectic topology.

The natural question being at the origin of \cite{LeC2}, is the following: can Theorem  \ref{th:principal} be directly deduced from the {\it foliated version of Brouwer Translation Theorem} ?  The preliminary lemma that is given here permits to answer positively to this question, at least to prove the existence of a fixed point.

Let us explain now the plan of the article. We will recall in the next section the principal results
of Brouwer Theory and will define precisely the many objects introduced above. In a short Section 3 we will introduce definitions related to ordered sets that will be used throughout the text. In Section 4 we will recall classical results about brick decompositions on surfaces. The introductory lemma will be stated and proved in Section 5, and it will be used in Section 6 to construct a nice brick decomposition. The proof of Theorem \ref{th:principal} will be done in Section 7. Some supplementary comments will be done in Section 8. In particular the links with the foliated version of Brouwer Translation Theorem will be explained there. We have tried to make the article the much self-contained as possible, adding precise references for the most technical details. 

\bigskip

\section{Recurrent homeomorphisms and Brouwer Theory} \label{section:recurrent homeomorphisms}

Usually, Brouwer Theory concerns homeomorphisms of $\C$ but can be translated immediately to homeomorphisms of a topological set homeomorphic to $\C$. As we are interested in homeomorphisms of $\D$ in this article, we will state the results in $\D$.

An {\it open disk} of a topological surface $S$ is a subset $V\subset S$ homeomorphic to $\D$. A {\it path} is a continuous map $\gamma: I\to S$, where $I$ is a non trivial real interval. Suppose that $\gamma$ is injective and proper: it is a {\it segment} if $I$ is compact, a {\it line} if $I$ is open, and a {\it half line} if $I=[a,+\infty)$ or $I=(-\infty,a]$. As usual, we will often identify a path and its image.

\subsection{Recurrent homeomorphisms} \label{subsection:recurrent}Let $S$ be a topological surface and $f$ a homeomorphism of $S$. A subset $X\subset S$ is {\it $f$-free} if $f(X)\cap X=\emptyset$ (we will say simply {\it free} if there is no ambiguity). 
Following \cite{LeC2}, we say that $f$ is {\it recurrent} if there exists a family  
of pairwise disjoint free open disks $(V_j)_{j\in\Z/r\Z}$, $r\geq 1$, and a family of positive integers $(k_j)_{i\in \Z/r\Z}$ such that $f^{k_i}(V_j)\cap V_{j+1}\not=\emptyset$, for every
$j\in\Z/r\Z$ (such a family is usually called a {\it free disks closed chain of length $r$}, see \cite{Fr1}). This property is clearly invariant by conjugacy. A simple example of a recurrent homeomorphism of $S$ is given by a homeomorphism that contains a non wandering free open disk, meaning a free open disk $V$ such that there exists $k>0$ such that $f^k(V)\cap V\not=\emptyset$. In particular, if the set of non wandering points is larger than the fixed point set, then $f$ is recurrent.  There exist more surprising examples of recurrent homeomorphisms, like the map $f: z\mapsto z/2$ defined on $\C$. Indeed the segment $\gamma: t\mapsto (8\pi -6\pi t)e^{2i\pi t}$ defined on $[0,1]$ is free but it holds that $f^2(\gamma)\cap \gamma\not=\emptyset$. So, one can find a free open disk $V$ that contains $\gamma$ and one knows that $f^2(V)\cap V\not=\emptyset$.

The fundamental property about recurrent homeomorphisms noticed by Franks in \cite{Fr1} is the fact that a periodic orbit can be obtained by continuous deformation, without adding any fixed point. Indeed, suppose that $f$ is recurrent and write $\mathrm{fix}(f)$ for the fixed point set. One can find an integer $r$ such that there exists a free disks closed chain of length $r$ but no free disks closed chain of length $<r$. Fix such a chain $(V_j)_{j\in\Z/r\Z}$. There exists a sequence of positive integers $(k_j)_{j\in\Z/r\Z}$, uniquely defined, such that:

\begin{itemize}
\item $f^{k_j}(V_j)\cap V_{j+1}\not=\emptyset$;
\item $f^{k}(V_j)\cap V_{j+1}=\emptyset$, if $1<k<k_j$.
\end{itemize}
 Fix $z_j\in V_j\cap f^{-k_j}(V_{j+1})$ and then consider a homeomorphism $h$ of $S$ supported on $\bigcup_{j\in\Z/r\Z} V_j$ such that $h(f^{k_j}(z_i))=z_{j+1}$, for every $j\in \Z/r\Z$. The minimality of $r$ implies that the only disk $V_l$, $l\in \Z/r\Z$, met by $\bigcup_{k\geq 1} f^k(V_j)$ is $V_{j+1}$, and so $h\circ f$ admits a periodic point of period $\sum_{j\in\Z/r\Z} k_j>1$. Moreover it holds that $\mathrm{fix}(h\circ f)=\mathrm{fix}(f)$. In fact, one can find an isotopy $(h_t)_{t\in[0,1]}$ joining $\mathrm{Id}$ to $h$ such that every $h_t$ is supported on $\bigcup_{j\in\Z/r\Z} V_j$  and so it holds that $\mathrm{fix}(h_t\circ f)=\mathrm{fix}(f)$, for every $t\in[0,1]$.

\subsection{Brouwer Theory} If $\lambda$ is an oriented line of $\D$, denote $R(\lambda)$ the connected component of $\D\setminus\lambda$, lying on the right of $\lambda$, and $L(\lambda)$ the component lying on its left. If $f$ is an orientation preserving homeomorphism of $\D$, say that $\lambda$ is a \textit{Brouwer line} of $f$  if
$f(\overline{L(\lambda)})\subset L(\lambda)$
or equivalently if
$f^{-1}(\overline {R(\lambda)})\subset R(\lambda).$

The Brouwer Translation Theorem (see \cite{Br}), asserts that if $f$ is a fixed point free orientation preserving homeomorphism of $\D$ (recall that $f$ is called a Brouwer homeomorphism), then $\D$ can be covered with Brouwer lines: every point $z\in\D$ belongs to such a line. In particular the set $$W=\bigcup_{k\geq 0} f^k(R(\lambda)) \cap \bigcup_{k\leq 0} f^{-k}(L(\lambda))$$ is an open set invariant par $f$ and homeomorphic to $\C$ and the map $f_{\vert W}$ is conjugate to a non trivial translation of $\C$. As a consequence, we know that every point $z$ of $\D$ is wandering, meaning that there is a neighborhood $U$ of $z$ such that the $f^k(U)$, $k\in\Z$, are pairwise disjoint. The proof of Brouwer Translation Theorem uses a preliminary result, telling that if an orientation preserving homeomorphism $f$ of $\D$  has a periodic point of period $>1$, then there exists a simple loop $\Gamma\subset \D\setminus\mathrm{fix}(f)$ such that $i(f,\Gamma)=1$ and consequently $f$ has a fixed point. In fact this preliminary result can be 
expressed in terms of translation arcs (see \cite{Br}, \cite{Fa} or \cite{G}). A segment $\gamma:[0,1]\to\D$ is a {\it translation arc} of $f$ if it joins a point $z\not\in\mathrm{fix}(f)$ to $f(z)$  and if:
\begin {itemize}
\item either $\gamma\cap f(\gamma)= \{z\}$;
\item or  $\gamma\cap f(\gamma)= \{z,f(z)\}$ and $f^2(z)=z$.
\end{itemize}

\bigskip
\bigskip

\centerline{{\includegraphics[scale=0.5]{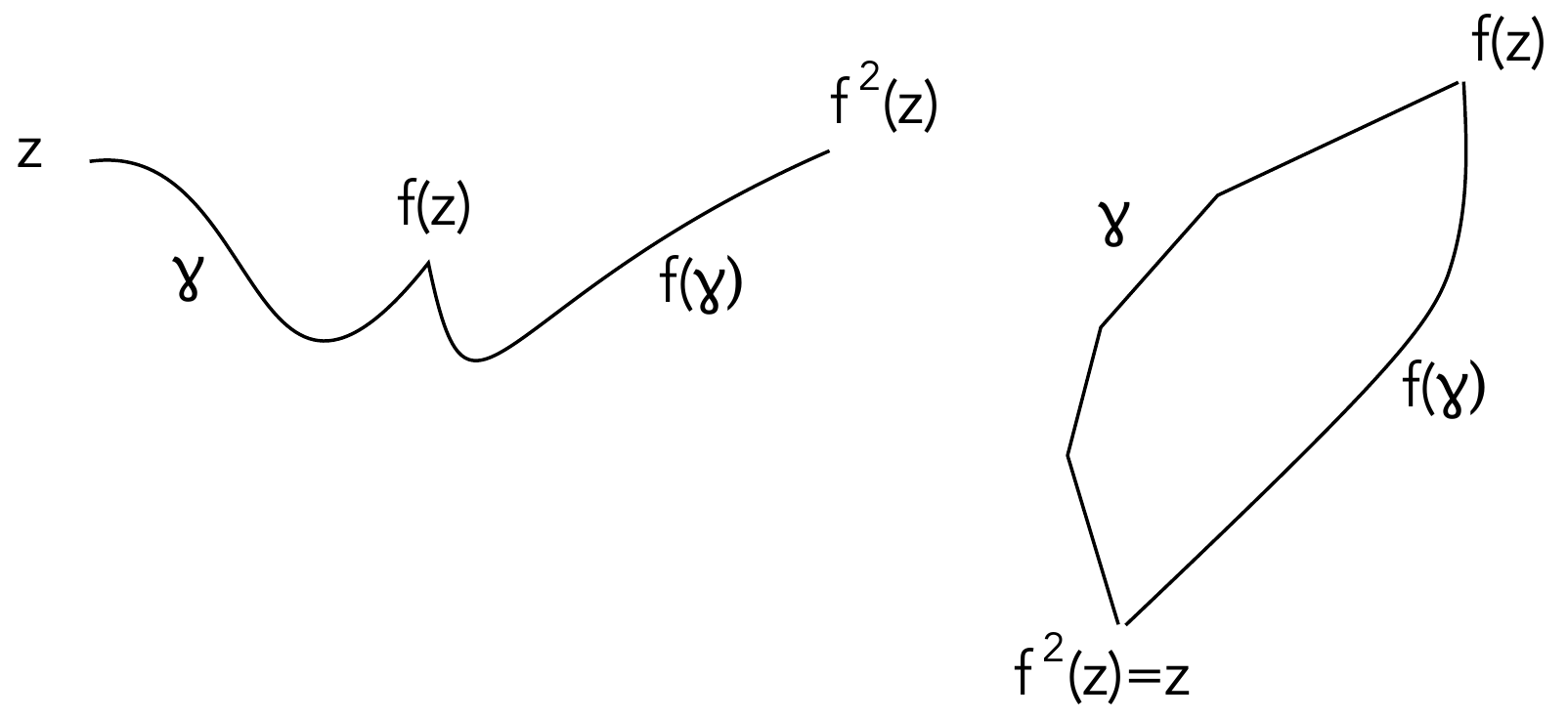}}}

\bigskip

\begin{proposition} \label{prop:brouwer} {Let $ f$ be an orientation preserving homeomorphism of $\D$ and $\gamma$ a translation arc of $f$ joining a point $z$ to $f(z)$. If $\gamma\cap f(\gamma)$ is not reduced to $f(z)$ or if there exists $k>1$ such that  $\gamma\cap f^k(\gamma)\not=\emptyset$, then one can construct a simple loop $\Gamma\subset \D\setminus\mathrm{fix}(f)$ such that $i(f,\Gamma)=1$. }
\end{proposition}

It is not difficult to construct a translation arc $\gamma$ passing through a given point $z$ (we will see the construction in Section \ref {section:translation arcs}). If $z$ is periodic, then $\gamma$ satisfies the hypothesis of Proposition \ref{prop:brouwer} and so, there exists a simple loop $\Gamma\subset \D\setminus\mathrm{fix}(f)$ such that $i(f,\Gamma)=1$.

Suppose that the hypothesis of Proposition \ref{prop:brouwer} are satisfied. If $z$ is periodic, then one can find a free open disk $V$  containing $z$, and one knows that there exists $k>0$ such that $f^k(V)\cap V\not=\emptyset$. If $z$ is not periodic,  there exists $k>1$ such that  $\gamma\cap f^k(\gamma)\not=\emptyset$. 
Fix $z'\in\gamma\cap f^{-k}(\gamma)$ and denote $\gamma'$ the sub-segment of $\gamma$ that joins $z'$ to $f^k(z')$. It is different from $\gamma$ because $z$ is not periodic and consequently it is free. So, one can find a free open disk $V$ that contains $\gamma'$ and one knows that $f^k(V)\cap V\not=\emptyset$. In all cases, there exists a non wandering free open disk and so $f$ is recurrent.

Let us recall now the following result of Franks \cite{Fr1}:

\begin{proposition} \label{prop:franks} {If $ f$ is an orientation preserving recurrent homeomorphism of $\D$, then there exists a simple loop $\Gamma\subset \D\setminus\mathrm{fix}(f)$ such that $i(f,\Gamma)=1$. Consequently, $f$ has at least one fixed point. }
\end{proposition}

Indeed, as seen in Section \ref{subsection:recurrent}, there exists a continuous family $(h_t)_{t\in[0,1]}$ of homeomorphisms such that:
\begin{itemize}
\item $h_t=\mathrm{Id}$;
\item $\mathrm{fix}(h_t\circ f) =\mathrm{fix}(f)$ for every $t\in[0,1]$;
\item $h_1\circ f$ has a periodic point of period $>1$. 
\end{itemize}
By Proposition  \ref{prop:brouwer},  there exists a simple loop $\Gamma\subset \D\setminus\mathrm{fix}(h_1\circ f)$ that does not contain any fixed point of $h_1\circ f$,  such that $i(h_1\circ f,\Gamma)=1$. The integer $i(h_t\circ f,\Gamma)$ is well defined for every $t\in[0,1]$, because $\mathrm{fix}(h_t\circ f) =\mathrm{fix}(f)$ for every $t\in[0,1]$, and depends continuously on $t$. So it  is independent of $t$, and we have $i( f,\Gamma)=1$.

Let us conclude with a characterization of the recurrence due to Guillou and Le Roux (see \cite{LeR1}, page 39). If $\Sigma$ is a submanifold of $\D$ (with or without boundary, of dimension $0$, $1$, or $2$) we will denote $\partial \Sigma$ the boundary of $\Sigma$ and define $\mathrm{int}(\Sigma)=\Sigma\setminus \partial \Sigma$.

\begin{proposition} \label{prop:guillouleroux} {Let $ f$ be an orientation preserving homeomorphism of $\D$. Suppose that there exists a family  $(\Sigma_j)_{j\in\Z/r\Z}$
of free connected submanifolds of $\D$ such that the sets $\mathrm{int}(\Sigma_j)$, $j\in\Z/r\Z$, are pairwise disjoint, and a family of positive integers $(k_j)_{j\in \Z/r\Z}$ such that $f^{k_j}(\Sigma_j)\cap \Sigma_{j+1}\not=\emptyset$, for every
$j\in\Z/r\Z$. Then $f$ is recurrent.}
\end{proposition}

Here again, recall the idea of the proof. Suppose that $(\Sigma_j)_{j\in\Z/r\Z}$ is a family 
of submanifolds satisfying the hypothesis and that $r$ is minimal, among all such families. In particular, the only manifold $\Sigma_l$, $l\in \Z/r\Z$, met by $\bigcup_{k\geq 1} f^k(\Sigma_j)$ is $\Sigma_{j+1}$.  Let $(k_j)_{j\in \Z/r\Z}$ be the family of integers given by the hypothesis, chosen the smallest as possible. Choose $z_j\in \Sigma_j\cap f^{-k_j}(\Sigma_{j+1})$. One can find a path $\gamma_j\subset \Sigma_j$ joining $f^{k_{j-1}}(z_{j-1})$ to $z_j$ such that $\gamma_j\setminus \{f^{k_{j-1}}(z_{j-1}),z_j\}\subset \mathrm{int}(\Sigma_j)$. This implies that the $\gamma_j$ are pairwise disjoint. Note that they are free and that $f^{k_j}(\gamma_j)\cap \gamma_{j+1}\not=\emptyset$, for every $j\in\Z/r\Z$. One can find a family of pairwise disjoint free open disks $(V_j)_{j\in\Z/r\Z}$ such that $V_j$ contains $\gamma_j$. We have proved that $f$ is recurrent because  $f^{k_j}(V_j)\cap V_{j+1}\not=\emptyset$.

\bigskip

\section{Ordered sets and cuts} \label{section:order}

We will consider different orders in the article, so let us introduce some definitions and notations that will be used later.

Let $(X, \leq)$ be an ordered set. For every $x\in B$, we denote 
$$x_{\leq}=\{y\in X\,\vert\, y\leq x\},$$ and define similarly the sets $x_{\geq}$, $x_{<}$ and $x_{>}$.

A {\it cut} of $\leq$ is a partition $c=(X_c^-, X_c^+)$ of $X$ in two sets $X_c^-$ and $ X_c^+$ (possibly empty), such that
$$x\in X_c^-\Rightarrow \,x_{\leq}\subset X_c^-\,, \enskip x\in X_c^+\Rightarrow \,x_{\geq}\subset X_c^+.$$ 

Denote $\mathcal C$ the set of cuts. The order $\leq$ admits a natural extension on $X\sqcup{\mathcal C}$ by the relations:
\begin{itemize}
\item $c\leq c'$ if  $X_c^-\subset X_{c'}^-$;
\item $x\leq c$ if $ x\in X_c^-$;
\item $x\geq c$ if  $ x\in X_c^+$.
\end{itemize}
Note that the cut $(\emptyset,X)$ is the smallest element of  $X\sqcup{\mathcal C}$ and that $(X,\emptyset)$ is the largest. Note also that if $\leq$ is a total order, then its natural extension is a total order. 

To conclude this section, recall that if $(X, \leq)$ is an ordered set, there exists a total order $\preceq$ on $X$, which is {\it weaker} than $\leq$, meaning that for every $x$, $x'$ in $X$, we have:
$$x\leq x'\Rightarrow x\preceq x'.$$ 
In particular every cut of $\preceq$ is a cut of $\leq$. Indeed, one can put an order $\ll$ on the set of orders $\mathcal O$, writing $\leq_1\,\ll\,\leq_2$ if $\leq_2$ is weaker than $\leq_1$. If $(\leq_{j})_{j\in J}$ is a totally ordered family in $\mathcal O$, then one gets an upper bound $\leq$ of this family, defining $x\leq y$ if there exists $j\in J$ such that $x\leq_{j} y$. By Zorn's lemma one can find an order $\preceq$, maximal for the order $\ll$,  which is weaker than $\leq$. It remains to prove that $\preceq$ is a total order. Argue by contradiction and suppose that there exist $x_0$ and $y_0$ that are non comparable. One gets an order $\preceq'$ strictly weaker than $\preceq$ by defining $x\preceq' y$ if $x\preceq y$ or if $x\preceq x_0$ and $y_0\preceq y$, in contradiction with the maximality of $\preceq$.

\section{Brick decompositions}

In this section, we will recall some facts about brick decompositions (see Sauzet \cite{Sa} or \cite{LeC1}). The results of the last sub-section are the only ones which are new. Nevertheless it seems us useful to give briefly some proofs of known results to make this article the most self-contained as possible. 

 \subsection{Definitions}

A {\it brick decomposition} ${\mathcal D}=(V,E,B)$ on an orientable surface $S$ is defined by a locally finite graph $\Sigma({\mathcal D})$ of $S$, the {\it skeleton} of $\mathcal D$, such that any vertex is locally the extremity of
exactly three edges. Here $V$ is the set of vertices, $E$ the set of edges (meaning the closures in $S$ of the connected components of $\Sigma({\mathcal D})\setminus V$) and $B$ the set of bricks  (meaning the closures in $S$ of the connected components of $S\setminus \Sigma({\mathcal D})$). Note that every edge is the image of a proper 
topological embedding of $[0,1]$,
$[0,+\infty)$,
$\R$ or $\S$ and that every brick is a surface (usually with boundary).

\bigskip
\bigskip

\centerline{{\includegraphics[scale=0.5]{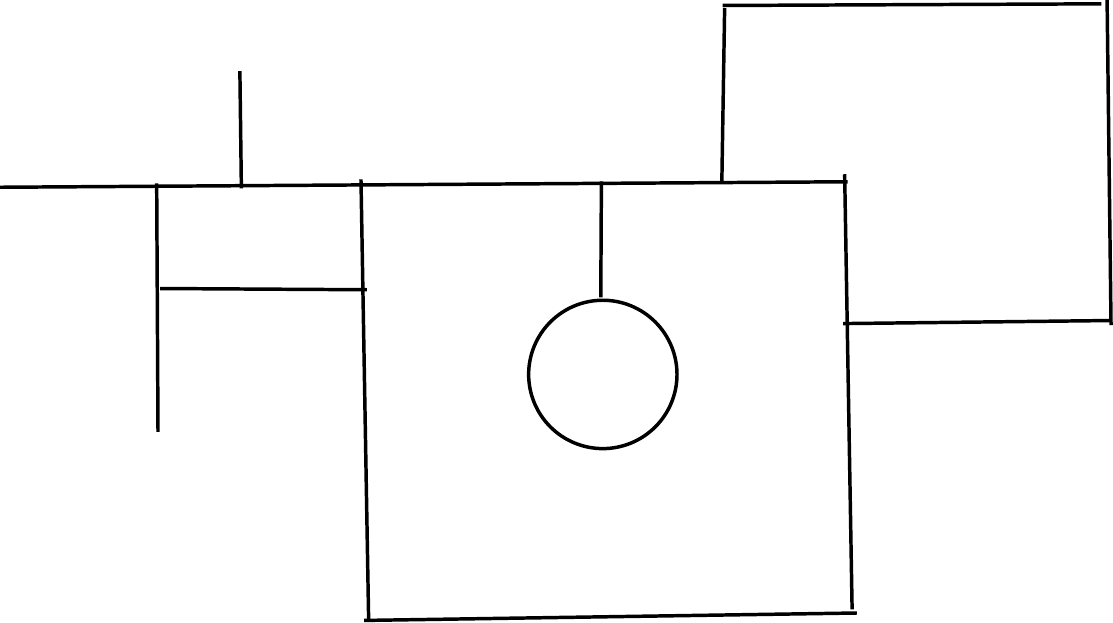}}}

\bigskip
In fact, for every 
$X\in \mathcal {P}(B)$, the union of
bricks $b\in X$ is a sub-surface of $S$ with a boundary contained in $\Sigma({\mathcal D})$.  It is
connected if and only if for every pair of bricks
$b, b'\in X$, there exists a sequence
$(b_j)_{0\leq j\leq n}$ such that:
\begin{itemize}
\item $b_0=b$ and 
$b_n=b'$;
\item $b_j$ and $b_{j+1}$ are {\it adjacent} for
$i\in\{0,\dots,n-1\}$ (meaning contain a common edge). 
\end{itemize}
The maximal connected subsets of a set $X\in \mathcal {P}(B)$ will be called the {\it connected components} of $X$\footnote{Connected components will usually simply be called components in the text.}
For 
simplicity we will denote by the same
letter a set of bricks and the union of these bricks, and will write $X\subset B$ or $X\subset S$ if specification is needed. We will say that two disjoint sets $X$ and $Y$ of $B$ are {\it adjacent} if $\partial X\cap \partial Y\not=\emptyset$ or equivalently if there is an edge $e\in E$ that is contained in a brick of $X$ and in a brick of $Y$.

\medskip
A brick decomposition ${\mathcal D}'=(V',E',B')$ is a {\it sub-decomposition} of ${\mathcal D}$ if $\Sigma({\mathcal D}')\subset\Sigma({\mathcal D})$. In that case $V'$ is a subset of $V$, every edge in $E'$ a union of edges in $E$ and every brick in $B'$ a union of bricks in $B$. Observe that one may 
have $B'=B$ even if ${\mathcal
D}'\not={\mathcal D}$. 
\bigskip

\subsection{The induced maps} Let $f$ be a homeomorphism of $S$ and ${\mathcal D}=(V,E,B)$ a brick decomposition of $S$. If $X\subset B$, then $f(X)\subset S$ is well defined. To define $\varphi_+(X)\subset B$, it is sufficient to set:
$$\begin{aligned}\varphi_+(X)&=\{b\in B\enskip\vert\enskip {\rm there 
\enskip exists}\enskip b'\in
X\enskip{\rm such\enskip that}\enskip b\cap f(b')\not=\emptyset\}\\
&=\{b\in B\enskip\vert\enskip b\cap f(X)\not=\emptyset\}.
\end{aligned}$$
It is easy to see that the map  $\varphi_+: \mathcal P(B)\to {\mathcal P}(B)$ sends connected subsets into connected subsets and satisfies
$$\varphi_+\left(\bigcup_{j\in J} X_j\right)=\bigcup_{j\in J}
\varphi_+(X_j),\enskip\enskip \varphi_+\left(\bigcap_{j\in J}
X_i\right)\subset\bigcap_{j\in J}
\varphi_+(X_j),$$ for every family $(X_j)_{j\in J}$ of subsets of 
$B$. 

\medskip
Similarly, one defines $\varphi_-: {\mathcal P}(B)\to {\mathcal
P}(B)$ writing~:
$$\eqalign{ \varphi_-(X)&=\{b\in B\enskip\vert\enskip {\rm there 
\enskip exists}\enskip b'\in
X\enskip{\rm such\enskip that}\enskip b\cap 
f^{-1}(b')\not=\emptyset\}\cr
&=\{b\in B\enskip\vert\enskip b\cap 
f^{-1}(X)\not=\emptyset\}.\cr}$$
The relations $b'\in\varphi_+^n(\{b\})$ and $b\in\varphi_-^n(\{b'\})$ 
are equivalent. Saying that $b\in\varphi_+^n(\{b\})$, $n\geq 1$, means that there exists a sequence
$(b_j)_{0\leq j\leq n}$ in $B$ such that:
\begin{itemize}
\item $b_0=b_n=b$ ;
\item $f(b_j)\cap b_{j+1}\not=\emptyset$ for every
$i\in\{0,\dots,n-1\}$. 
\end{itemize}
In that case we have a {\it cycle of bricks}.

\subsection{Free brick decompositions, the induced order} Let $f$ be a fixed point free homeomorphism of $S$. Say that ${\mathcal
D}=(V,E,B)$ is a {\it free brick decomposition} of $f$ if every brick $b\in B$ is free, meaning that $f(b)\cap b=\emptyset$. 
Such a decomposition always exists. More precisely, if ${\mathcal
D'}=(V',E',B')$ is a brick decomposition of $S$, there exists a free brick decomposition  ${\mathcal
D}=(V,E,B)$ such that  ${\mathcal
D'}$  is a sub-decomposition of ${\mathcal
D}$. Let ${\mathcal
D}=(V,E,B)$ be a free brick decomposition. An easy application of Zorn's lemma tells us that there exists a partition $(X_j)_{j\in J}$ of $B$ into free connected subsets, such that if $j\not=j'$ then $X_j\cup X_{j'}$ is not free if connected. The
set $\displaystyle\bigcup_{j\in
J}\partial X_j$ is the skeleton of a sub-decomposition ${\mathcal D}'$ of 
$\mathcal D$ whose bricks are the
$X_j$ and such that every
edge is contained in two bricks (and not one) of the decomposition. We obtain in that way a {\it maximal free brick decomposition}: it is free and every strict sub-decomposition of
${\mathcal D}'$ is not free.

Suppose now that $f$ is a	 non recurrent fixed point free homeomorphism of $S$ and that ${\mathcal
D}=(V,E,B)$ is a free brick decomposition of $f$. Applying Proposition  \ref{prop:guillouleroux}, one deduces that there is no cycle of bricks.  Consequently, one gets an order $\leq$ 
on $B$, writing~:
$$b\leq b'\enskip\Longleftrightarrow \enskip b'\in\bigcup_{ n\geq 0	
}\varphi_+^n(\{b\}).$$  In other words, $b\leq b'$ if and only if there exists a sequence
$(b_j)_{0\leq j\leq n}$, $n\geq 0$, in $B$ such that:
\begin{itemize}
\item $b_0=b$ and $b'_n=b'$ ;
\item $f(b_j)\cap b_{j+1}\not=\emptyset$ for every
$j\in\{0,\dots,n-1\}$. 
\end{itemize}

In the case where ${\mathcal
D}$ is maximal, then two adjacent bricks $b$ and $b'$ are comparable because $b\cup b'$ is not free. More precisely, there is a natural 
orientation of
$\Sigma(D)$ defined as follows: every edge
$e\in E$ is oriented in such a way that $f(r(e))\cap l(e)\not=\emptyset$,
where
$r(e)$ and
$l(e)$ are the two bricks containing $e$, the first one on the right, the second one on the left. Note that a vertex is neither
the final point of three edges nor the initial point of three edges (otherwise there will be a cycle of bricks of length $3$). 

\begin{lemma} \label{lemma:withnofixedpoints} Let $f$ be a non recurrent fixed point free homeomorphism of $S$ and ${\mathcal
D}=(V,E,B)$ be a maximal free brick decomposition of $f$. The following conditions are equivalent, for a given brick $b\in B$:

\begin{enumerate}

\item there exists $b'$ adjacent to $b$ such that $f(b)\cap b'\not=\emptyset$,

\item $b_{\geq}$ is connected;
\end{enumerate}
and they imply that $b_{>}$ is connected. In the case where these conditions are satisfied, we will say that {\em $b$ is positively regular.} Otherwise we will say that 
 {\em $b$ is positively singular.} 

\end {lemma}
\begin{proof} The inclusion $(2)\Rightarrow (1)$ is obvious. Indeed, $b_{\geq}$ contains $b$ and is not reduced to $b$, so its contains a brick $b'$ adjacent to $b$ if it is connected. Of course, $f(b')\cap b=\emptyset$ because $b'\geq b$ and so we have $f(b)\cap b'\not=\emptyset$. 
To prove the inclusion $(1)\Rightarrow (2)$, note that $\varphi_+(\{b\})$ is connected for every $b\in B$ and so $\{b\}\cup\varphi_+(\{b\})$ is also connected if $(1)$ is satisfied. Consequently, each set $$\varphi_+^n(\{b\})\cup\varphi_+^{n+1}(\{b\})=\varphi_+^n(\{b\}\cup\varphi_+(\{b\})),$$  is connected. Noting that $\varphi_+^n(\{b\})\cup\varphi_+^{n+1}(\{b\})$ and  $\varphi_+^{n+1}(\{b\})\cup\varphi_+^{n+2}(\{b\})$ intersect and writing
$$b_{\geq}=\bigcup_{n\geq 0}  \varphi_+^n(\{b\})= \bigcup_{n\geq 0} \varphi_+^n(\{b\})\cup\varphi_+^{n+1}(\{b\}), $$
we deduce that $b_{\geq}$ is connected. We also deduce that $b_{>}$ is connected because
$$b_{>}=\bigcup_{n\geq 1}  \varphi_+^n(\{b\})= \bigcup_{n\geq 1} \varphi_+^n(\{b\})\cup\varphi_+^{n+1}(\{b\}). $$
\end{proof}

 Similarly we will say that  $b$ is {\it negatively regular} if $b_{\leq}$ is connected, meaning that there is a brick $b'$ adjacent to $b$ such that $b'\leq b$. Otherwise it will be called {\it negatively singular}. We will say that $b$ is {\it regular} if it is both positively and negatively regular and {\it singular} otherwise.
 
\subsection{Attractors and repellers}

Let $f$ be a homeomorphism of $S$ and ${\mathcal
D}=(V,E,B)$ a  brick decomposition of $S$. Let us call {\it attractor} a set $A\subset B$ such that 
$\varphi_+(A)\subset A$ or equivalently such that $f(A)\subset \mathrm{int} (A)$. In particular, if $(A_j)_{j\in J}$ is a family of attractors, then $\bigcup_{j\in J}
A_j$ and $\bigcap_{j\in J}
A_j$ are attractors.

 Of course, if $A$ is an attractor, $\varphi_+$ sends components of $A$ into components. A component $C$ satisfying $\varphi_+(C)\subset C$ is called {\it regular}, a component $C$ satisfying $\varphi_+(C)\cap C=\emptyset$ is called {\it singular}. 
 
 In the case where $f$ is a non recurrent fixed point free homeomorphism of $S$ and ${\mathcal
D}=(V,E,B)$ is a maximal free brick decomposition of $S$, then $C$ is regular if and only if contains at least two bricks, and singular if and only if it is reduced to a positively singular brick. Indeed, let $C$ be a component of $A$ and denote $C'$ the component of $A$ containing $\varphi_+(C)$. If $\# C\geq 2$, then $C$ contains two adjacent bricks and so $\varphi_+(C)\cap C\not=\emptyset$. Consequently one has $C'=C$.   If $C$ is reduced to a brick $b$, then $C'\not=C$ because $b$ is free. Moreover, $C\cup C'$ is not connected, which implies that $\varphi_+(\{b\})$ does not contain a brick adjacent to $b$. So $b$ is positively singular. 

Similarly, a {\it repeller} is a set $R\subset B$ such that
$\varphi_-(R)\subset R$ or equivalently such that $f^{-1}(R)\subset \mathrm{int} (R)$. We have a similar classification for its components. Note also that $R$ is a repeller if and only if $B\setminus R$ is an attractor. In fact, a couple $(R,A)$ such that $\{R,A\}$ is a partition of $B$ that consists of a repeller $R$ and an attractor $A$, is  a cut of $\leq$. Observe that $\partial R=\partial A$.

 \subsection{New and useful results}

We will finish this section by proving some new results on brick decompositions that will be useful in our purpose. We suppose that $f$ is a non recurrent orientation preserving homeomorphism of the $2$-sphere $\S^2$ and that ${\mathcal
D}=(V,E,B)$ is a maximal free brick decomposition of $f_{\vert \S^2\setminus\mathrm{fix}(f)}$. We denote $\leq$ the induced order on $B$ and $\varphi_+$, $\varphi_-$ the induced maps on ${\mathcal P}(B)$\footnote{Of course, by considering its extension to the Alexandrov compactification of $\D$, everything will work for a homeomorphism $f$ of $\D$ if we replace $\S^2\setminus \mathrm{fix}(f)$ with $\D\setminus \mathrm{fix}(f)$. }. 

We begin with the following lemma:

\begin{lemma} \label{lemma:singular brick}  The boundary of a singular brick is a line of $\S^2\setminus\mathrm{fix}(f)$, whose closure in $\S^2$ meets a unique connected component of $\mathrm{fix}(f)$.  \end {lemma}
\begin{proof} It is sufficient to prove the lemma for negatively singular bricks. Let $b$ such a brick. Its boundary is a boundaryless $1$-submanifold of $\S^2\setminus\mathrm{fix}(f)$ and each of its component is 
\begin{itemize}
\item  a simple loop,
\item  or a line of $\S^2\setminus\mathrm{fix}(f)$ whose closure in $\S^2$ meets a unique connected component $K$ of $\mathrm{fix}(f)$,
\item  or a line of $\S^2\setminus\mathrm{fix}(f)$ whose closure in $\S^2$ meets exactly two connected component $K_-$ and $K_+$ of $\mathrm{fix}(f)$.
\end{itemize}

Let us prove first that if there exists a component $\Gamma$ of $\partial b$ which is a line of $\S^2\setminus\mathrm{fix}(f)$ whose closure in $\S^2$ meets exactly one connected component $K$ of $\mathrm{fix}(f)$, then  $\partial b=\Gamma$. The line $\Gamma$ is contained in a component $W$ of $\S^2\setminus K$, this component is homeomorphic to $\D$ and $\Gamma$ separates $W$.  We denote $U$ the connected component of $W\setminus\Gamma$ that contains $\mathrm{int}(b)$ and $V$ the other one. The bricks $b'\not=b$ adjacent to $\Gamma$ are included in $\overline V$.  The brick $b$ being negatively singular, $f(b)$ meets every brick adjacent to $b$. Moreover $f(b)$ is connected and disjoint from $\Gamma$ because $b$ is free.  So every brick adjacent to $b$ is in $\overline V$. This implies that $\partial b$ has no component but $\Gamma$.

Let us prove now that no component of $\partial b$ is a loop. We argue by contradiction and suppose that such a component $\Gamma$ exists. It separates $\S^2$. We denote $U$ the connected component of $\S^2\setminus\Gamma$ that contains $\mathrm{int}(b)$ and $V$ the other one. The same proof as above tells us that every brick adjacent to $b$ is in $\overline V$. One deduces that $b=\overline U$. In this situation, the set $X$ of bricks adjacent to  $b$ is finite. For every $b'\in X$, the set $f^{-1}(b')$ meets $b$ but is not included in $b$ (because the image by $f^{-1}$ of an edge $e\subset b\cap b'$ is disjoint from $b$). Moreover $f^{-1}(b')$ is connected and so it meets another brick of $X$. We have proved that for every $b'\in X$ there exists $b''\in X$ such that $b''<b'$. This is not compatible with the finiteness of $X$.

Let us conclude by proving that the third case never occurs. We argue by contradiction and consider a component $\Gamma$ of $\partial b$ which is a line of $\S^2\setminus\mathrm{fix}(f)$ whose closure in $\S^2$ meets exactly two connected component $K_-$ and $K_+$ of $\mathrm{fix}(f)$. The line $\Gamma$ is contained in a component $W$ of $\S^2\setminus (K_-\cup K_+)$, this component is homeomorphic to the annulus $\S\times\R$ and $\Gamma$ does not separate $W$ but  joins its two ends. Consequently, $\Gamma$ is not the unique component of $\partial b$. The brick $b$ being positively regular,  Lemma \ref{lemma:withnofixedpoints} tells us that
$b\geq$ and $b_{>}=b_{\geq}\setminus\{b\}$ are connected. The brick $b$ being negatively singular, every brick adjacent to $b$ belongs to $b_{>}$. The connectedness of $b$ and $b_{>}$ implies that $b_{\geq}$ contains an essential simple loop $\Gamma'$, meaning non homotopic to zero in $W$.  Its image by $f$ is contained in the interior of  $b_{>}$ and consequently does not meet $\Gamma$. So it cannot be essential and we have got a contradiction.\end{proof}

\begin{corollary} \label{corollary:singular}  Let $(R,A)$ be a cut of $\leq$. If a singular component $R'$ of $R$ is adjacent to a component $A'$ of $A$, then $A'$ is the only component of $A$ that is adjacent to $R'$. \end{corollary}
\begin{proof} In that case, $R$ and $A$ have the same boundary and every component of this boundary is the boundary of a unique component of $R$ and of a unique component of $A$. By Lemma  \ref{lemma:singular brick}, the boundary of $R'$ is connected. It is a connected component of the boundary of $R$, and so is a connected component of  the boundary of a component $A'$ of $A$, moreover $A'$ is the unique component of $A$ that is adjacent to $R'$.\end{proof}

Of course, one proves similarly that if a singular component $A'$ of $A$ is adjacent to a component $R'$ of $R$, then $R'$ is the only component of $R$ that is adjacent to $A'$.

\begin{proposition} \label{proposition:components}  Let $(R,A)$ be a cut of $\leq$. If $R'$ is a regular component of $R$ and $A'$ a regular component of $A$ and if there exists $n\geq 1$ such that $f^n(R')\cap A'\not=\emptyset$, then $R'$ and $A'$ are adjacent. \end {proposition}
\begin{proof} 

Let $X$ be the union of $R'$ and of the components of $A$ that are adjacent to $R'$. It is a connected set and we want to prove that it contains $A'$. 
Let $e$ be an edge included in $\partial X$. Denote $b$ the brick containing $e$ and included in $X$ and $b'$ the other brick.  Of course $b\not\in R'$, otherwise $b'$ would belong to a component of $A$ adjacent to $R'$ and so would belong to $X$. So, $b$ belongs to a connected component $A''$ of $A$  that is adjacent to $R'$ and $b'$ is contained in a connected component of $R$ adjacent to $A''$ and different from $R'$. By Corollary  \ref{corollary:singular}, it implies that $A''$ is regular and we have $f(e)\subset f(A'')\subset \mathrm{int}( A'')\subset \mathrm{int}(X)$. So, $X$ is a connected set such that $f(\partial X)\subset  \mathrm{int}(X)$. Moreover, if  $A'$ is not adjacent to $R'$, then $A'\cap X=\emptyset$ in $B$ but also in $\S^2$. By hypothesis, the sequence 
$(f^{-k}(A'))_{k\geq 0}$ is increasing and for $k$ large enough it holds that $f^{-k}(A')\cap R'\not=\emptyset$. Consequently, there exists $k\geq 0$ such that $f^{-k}(A')\cap X=\emptyset$ and $f^{-k-1}(A')\cap X\not=\emptyset$. But $f^{-k-1}(A')$ is connected and not included in $X$. It implies that  $f^{-k-1}(A')\cap \partial X\not=\emptyset$, which implies that $f^{-k}(A')\cap \mathrm{int}(X)\not=\emptyset$. We have found a contradiction.\end{proof}

\section{A preliminary lemma} \label{section:translation arcs}

We begin from now on the proof of Theorem \ref{th:principal}. We suppose in this section that $f:\D\to\D$ satisfies the hypothesis of the theorem and is not recurrent. The final goal is to prove that there is a contradiction. For every $i\in\Z/p\Z$ and every $k\in\Z$ we define $z_i^k=f^k(z_i)$.

Let $(X^k)_{k\geq 0} $ be a sequence of subsets of $\D$ and $\zeta\in \S$. We will write $\lim_{k\to +\infty} X^k =\zeta$ if for every neighborhood $U$ of $\zeta$ in $\D\cup\{\zeta\}$, there exists $K\geq0$ such that $X^k\subset U$, for every $k\geq K$.

 Let us begin with a result whose detailed proof can be found in \cite{LeC2}:

\begin{lemma} \label{lemma:withnofixedpoints}There exists a sequence $(V_i^k)_{k\in\Z}$ of open disks such that:

\begin{enumerate}

\item $V_i^k$ contains $z_i^k$ and  $z_i^{k+1}$;

\item $V_i^k\cap\mathrm{fix}(f)=\emptyset$;

\item $\lim_{k\to -\infty} V_i^k =\alpha_i$;

\item $\lim_{k\to +\infty} V_i^k =\omega_i$.
\end{enumerate}

\end {lemma}

\begin{proof} By hypothesis, one can find a sequence $(U_i^k)_{k\in\Z}$ of open disks such that

\begin{itemize} 

\item $U_i^k$ contains $z_i^k$ and  $z_i^{k+1}$;

\item $\lim_{k\to -\infty} U_i^k =\alpha_i$;

\item $\lim_{k\to +\infty} U_i^k =\omega_i$.
\end{itemize}

Fix $i\in\Z/p\Z$ and $k\in\Z$ and choose a homeomorphism $h: \C\to U_i^k$ such that 
$$h(-1)=z_i^k,\enskip h(1)=z_i^{k+1}.$$

Consider the  following segments, defined on [0,1], that join $z_i^k$ to $z_i^{k+1}$:
$$\delta_1 : t\mapsto h\left( e^{i\pi (1+t)}\right),\enskip \delta _2:  t\mapsto h( -1+2t),\enskip \delta_3 : t\mapsto h\left( e^{i\pi (1-t)}\right).$$

\bigskip
\bigskip

\centerline{{\includegraphics[scale=0.5]{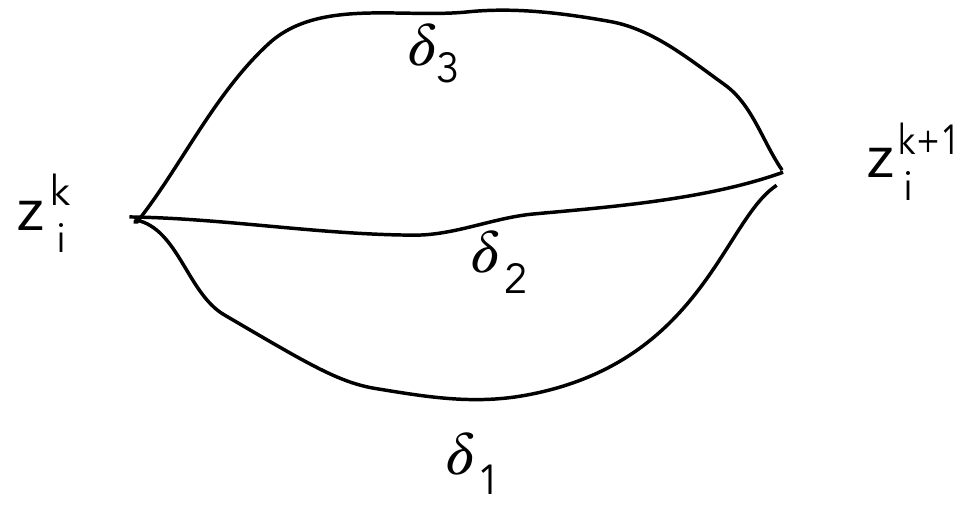}}}

\bigskip

In the case where there exists $j\in\{1,2,3\}$ such that $\delta_j\cap\mathrm{fix}(f)=\emptyset$, one can find an open disk $V_j^k\subset U_j^k$, close enough to $\delta_j$ to satisfy $V_j^k\cap\mathrm{fix}(f)=\emptyset$. In the case where every $\delta_j$ meets $\mathrm{fix}(f)$, define $\delta'_j=\delta_j{}_{\vert [0,t_j)}$, where $$t_j=\min\{t\in[0,1]\,\vert \,\delta_j(t)\in \mathrm{fix}(f)\}.$$

The fact that $f$ preserves the orientation implies that the cyclic order of the germs of $\delta_1$, $\delta_2$ and $\delta_3$ at $z_i^k$ coincides with the cyclic order of the germs of $f(\delta_1)$, $f(\delta_2)$ and $f(\delta_3)$ at $z_i^{k+1}$. This implies that there exist $j_1$ and $j_2$ in $\{1,2,3\}$ such that $f(\delta'_{j_1})\cap \delta'_{j_2} \not=\emptyset.$ Write  $$t'_{j_1}=\min\{t\in[0,t_{j_1})\,\vert \,f(\delta_j(t))\in \delta'_{j_2}\}$$ and define $t'_{j_2}\in [0, t_{j_2})$ such that $\delta_{j_2}(t'_{j_2})= f(\delta_{j_1}(t'_{j_1}))$. One can find an open disk $V_j^k\subset U_j^k\cup f(U_j^k)$, close enough to the segment $\delta_{j_2}([0, t'_{j_2}])\cup f(\delta_{j_1}([0, t'_{j_1}])$, to satisfy $V_j^k\cap\mathrm{fix}(f)=\emptyset$.

\bigskip

\centerline{{\includegraphics[scale=0.5]{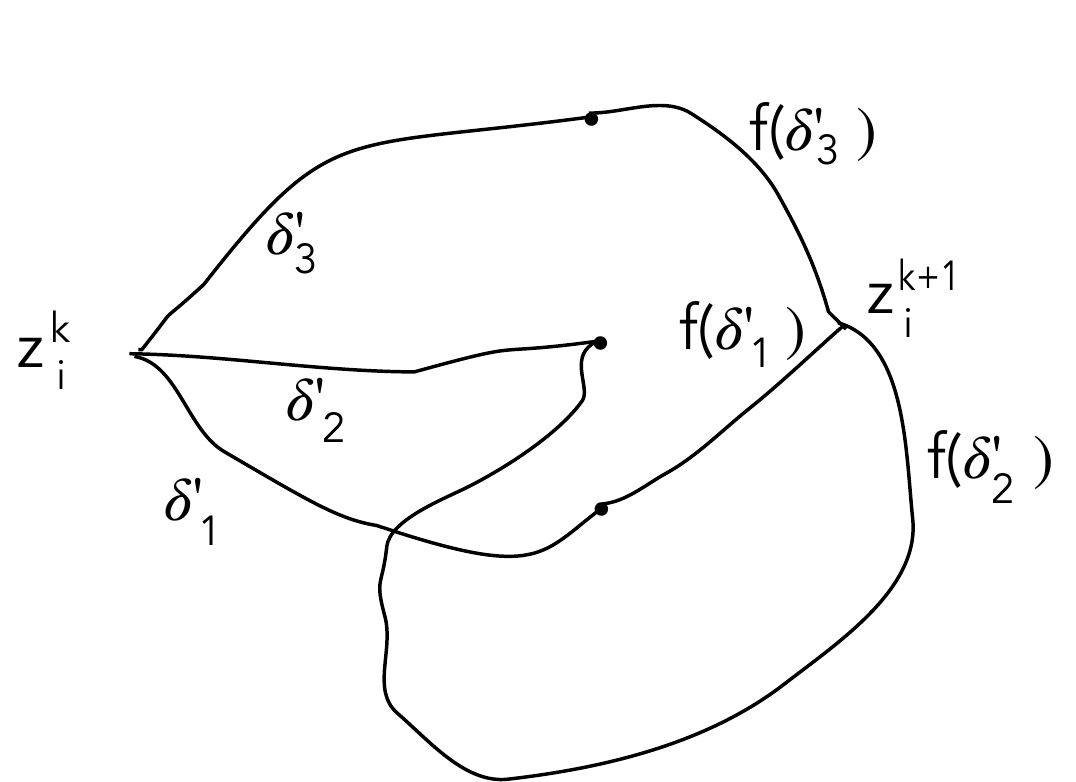}}}

\bigskip

The fact that $f$ extends to a homeomorphism of $\D\cup_{i\in\Z/p\Z} \{\alpha_i, \omega_i\}$ implies that
\begin{itemize}

\item $\lim_{k\to -\infty} f(U_i^k) =\alpha_i$,

\item $\lim_{k\to +\infty}  f(U_i^k) =\omega_i$;

\end{itemize}
and we deduce that

\begin{itemize}

\item $\lim_{k\to -\infty} V_i^k =\alpha_i$,

\item $\lim_{k\to +\infty} V_i^k =\omega_i$.

\end{itemize}

\end{proof}

Now, state the result which will be fundamental in our proof:

\begin{lemma} \label{lemma:new}There exists a sequence $(\gamma_i^k)_{k\in\Z}$ of translation arcs such that:

\begin{enumerate}

\item $\gamma_i^k$ joins $z_i^k$ to $z_i^{k+1}$;

\item $\gamma_i^{k-1}\cap \gamma_i^{k}$ coincide in a neighborhood of $z_i^k$;

\item $\lim_{k\to -\infty} \gamma_i^k =\alpha_i$;

\item $\lim_{k\to +\infty} \gamma_i^k =\omega_i$.
\end{enumerate}

\end {lemma}

\begin{proof} Let $(V_i^k)_{k\in\Z}$ be a sequence of open disks given by Lemma \ref{lemma:withnofixedpoints}. Fix $i\in\Z/p\Z$ and $k\in\Z$ and choose a homeomorphism $h: \D\to V_i^k$ such that $h(0)=z_i^k$.

For every $r<1$, write $D_r=h(\{z\in \D\,\vert \, \vert z\vert \leq r\})$. Then define  $D_* =D_{r_*}$, where $$r_*=\sup\{r>0\,\vert\, f( D_r)\cap  D_r=\emptyset\}.$$ In particular we have 
$$f(\mathrm{int}( D_*))\cap \mathrm{int}( D_*)=\emptyset \enskip \mathrm{and} \enskip f(\partial {D_*})\cap \partial {D_*}\not=\emptyset.$$ Now, fix a point $y_i^k\in f(\partial {D_*})\cap \partial {D_*}$. It is not fixed because $D_*$ is included in $V_i^k$.  So, one can  construct a segment $\sigma_i^k=\beta_i^k\delta_i^k$ joining $f^{-1}(y_i^k)$ to $y_i^k$, where:

\begin{itemize}
\item $\beta_i^k$ joins $f^{-1}(y_i^k)$ to  $z_i^k$  and $\delta_i^k$ joins $z_i^k$ to $y_i^k$;
\item   $\beta_i^k\setminus \{f^{-1}(y_i^k)\}$ and $\delta_i^k\setminus \{y_i^k\}$ are included in $\mathrm{int}( D_*)$.
\end{itemize}

We have $f^{-1}(y_i^k)\not = f(y_i^k)$ because $f$ is not recurrent and so $f(\sigma_i^k)\cap \sigma_i^k $ is reduced to $y_i^k$. Consequently, $\sigma_i^k$ is a translation arc. As explained in the previous section, the segments $f^r(\sigma_i^k)$  and $f^{r'}(\sigma_i^k)$ are disjoint if $\vert r-r'\vert >1$. In particular $\sigma_i^k \,f(\sigma_i^k)\,f^2(\sigma_i^k)$ is a segment. But this implies that $\gamma_i^k=\delta_i^k\, f(\beta_i^k)$ is a translation arc joining $z_i^k$ to $z_i^{k+1}$ that is contained in $D_*\cup f(D_*)$ and so is contained in $V_i^k\cup f(V_i^k)$. The fact that $f$ extends to a homeomorphism of $\D\setminus\{\alpha_i, \omega_i\}$ implies that
\begin{itemize}

\item $\lim_{k\to -\infty} f(V_i^k) =\alpha_i$,

\item $\lim_{k\to +\infty}  f(V_i^k) =\omega_i$;

\end{itemize}
and we deduce that
\begin{itemize}

\item $\lim_{k\to -\infty} \gamma_i^k =\alpha_i$;

\item $\lim_{k\to +\infty} \gamma_i^k =\omega_i$.

\end{itemize}

\bigskip

\centerline{{\includegraphics[scale=0.5]{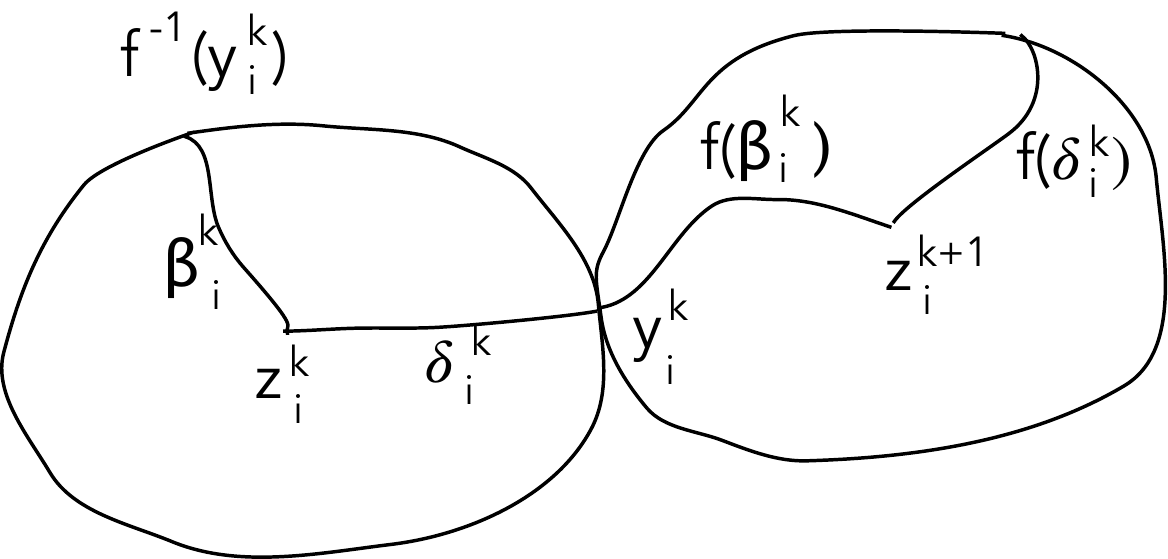}}}

\bigskip

To finish the proof it remains to explain why the previous construction can be done in such a way that  $\gamma_i^{k-1}$ and $\gamma_i^{k}$ coincide in a neighborhood of  $z_i^k$, which means that $f(\beta_i^{k-1})$ and $ \delta_i^{k}$ coincide in a neighborhood of $z_i^k$. The reason is that we have a lot of freedom when choosing the $\beta_i^k$ and the $\delta_i^k$. While choosing our paths $\beta_i^k$, $k\geq 0$, we can always suppose by an induction process that for every $k\geq 0$, the paths $f(\beta_i^k)$ and $\beta_i^{k+1}$ intersect only at $z_i^k$  in a neighborhood of this point. Then, again by induction, while constructing our paths $\beta_i^k$, $k<0$, one can suppose that the previous property is true for every $k\in\Z$. Finally, while constructing our paths $\delta_i^k$, $k\in\Z$,  there is no more obstruction to suppose  that $f(\beta_i^{k-1})$ and $ \delta_i^{k}$ coincide in a neighborhood of $z_i^k$. \end{proof}

\section{An adapted brick decomposition} \label{section:adapted}

Here again, we suppose that $f:\D\to\D$ satisfies the hypothesis of Theorem \ref{th:principal} and is not recurrent. We consider, for every $i\in\Z/p\Z$,  a sequence $(\gamma_i^k)_{k\in\Z}$ of segments satisfying the hypothesis of  Lemma  \ref{lemma:new}.

\subsection{Construction of critical graphs} Each $\gamma_i^k$ is a translation arc and so the path obtained by concatenation of all $f^{k'}(\gamma_i^k)$, $k'\in\Z$, is a simple path. In particular,  $z_i^{k'}\not\in \gamma_i^k$ if $k'\not\in\{k,k+1\}$. We parametrize every $\gamma_i^k$ on $[0,1]$. The fact that   $\lim_{k\to -\infty} \gamma_i^k =\alpha_i$  and $\lim_{k\to +\infty} \gamma_i^k =\omega_i$ implies that one can define 
$$k_0=\min\{k'\leq 0\,\vert\, \gamma_i^{k'}\cap \bigcup_{k>0}\gamma_i^{k}\not=\emptyset\}$$
and 
$$k_1=\max\{k>0\,\vert\, \gamma_i^{k}\cap \gamma_i^{k_0}\not=\emptyset\}.$$
We define
$$t_0=\min\{t\in[0,1]\,\vert\, \gamma_i^{k_0}(t)\in\gamma_i^{k_1}\}$$
and then $t_1\in[0,1]$ such that $\gamma_i^{k_0}(t_0)=\gamma_i^{k_1}(t_1)$, noting that 
$$t_1=\max\{t\in [0,1]\,\vert\, \gamma_i^{k_1}(t)\in  \gamma_i^{k_{0}}{}_{\vert [0,t_0]}\}.$$

The fact that $z_i^{k'}\not\in  \bigcup_{k>0}\gamma_i^{k}$ if $k'\leq 0$ implies that  $t_0\not=0$, and moreover that $t_0\not=1$ if $k_0<0$. For the same reasons, $t_0\not=1$, if $k_0=0$ and $k_1>1$. Finally $t_0\not=1$, if $k_0=0$ and $k_1=1$, because $\gamma_i^{0}$ and $ \gamma_i^{1}$ coincide in a neighborhood of $z_i^1$. Consequently, in every situation, it holds that $t_0\in(0,1)$. We deduce that $t_1\in(0,1)$ because $\gamma_i(t_0)$ does not belong to the orbit of $z_i$.

We define inductively an increasing sequence $(k_l)_{l>0}$ of positive integers and a sequence $(t_l)_{l>0}$ by the relations:
$$k_{l+1}=\max\{k>k_l\,\vert\, \gamma_i^{k}\cap\gamma_i^{k_l}{}_{\vert [t_l,1]}\not=\emptyset\},$$
$$t_{l+1}=\max\{t\in [0,1]\,\vert\, \gamma_i^{k_{l+1}}(t)\in  \gamma_i^{k_{l}}{}_{\vert  [t_l,1]}\}.$$

Of course, $\gamma_i^{k_{l}}(t_{l})\not=\gamma_i^{k_{l-1}}(t_{l-1})$ because $\gamma_i^{k_{l-1}}(t_{l-1})\not\in \gamma_i^{k_{l}}$. Note also that $t_l\not=0$ and more generally than $\gamma_i^{k_l}(t_l)$ is not in the orbit of $z_i$  because $z_i^{k_l}$ and $z_i^{k_l+1}$ are the only point of the orbit of $z_i$ that are contained in $\gamma_i^{k_l}$ and because  $\gamma_{i}^{k_{l}}$ and $\gamma_i^{k_l-1}$ coincide in a neighborhood of $z_i^{k_l}$

Now we define inductively an increasing sequence $(k_l)_{l<0}$ of negative integers and a sequences $(t_l)_{l\leq 0}$  by the relations:
$$k_{l-1}=\min\{k<k_l\,\vert\, \gamma_i^{k}\cap\gamma_i^{k_l}{}_{\vert [0,t_l]}\not=\emptyset\}$$
$$t_{l-1}=\min\{t\in [0,1]\,\vert\, \gamma_i^{k_{l-1}}(t)\in  \gamma_i^{k_{l}}{}_{\vert[0,t_l]}\},$$ 
Similarly, we prove that $t_l\in(0,1)$ and that $\gamma_i^{k_{l}}(t_{l})$ is not an end of $\gamma_i^{k_{l+1}}{}_{\vert[0,t_{l+1}]}$. 

Now we define a family of segments $(\gamma'{}_i^l)_{l\in\Z\setminus\{0\}}$ by the relations:
$$\gamma'{}_i^l=\begin{cases}  
 \gamma_i^{k_l}{}_{\vert [t_l,1] }&  \mathrm{if} \enskip  l>0\\
  \gamma_i^{k_l}{}_{\vert [0,t_l]} & \mathrm{if} \enskip l\leq 0 \\
\end{cases} $$

\bigskip

\centerline{{\includegraphics[scale=0.5]{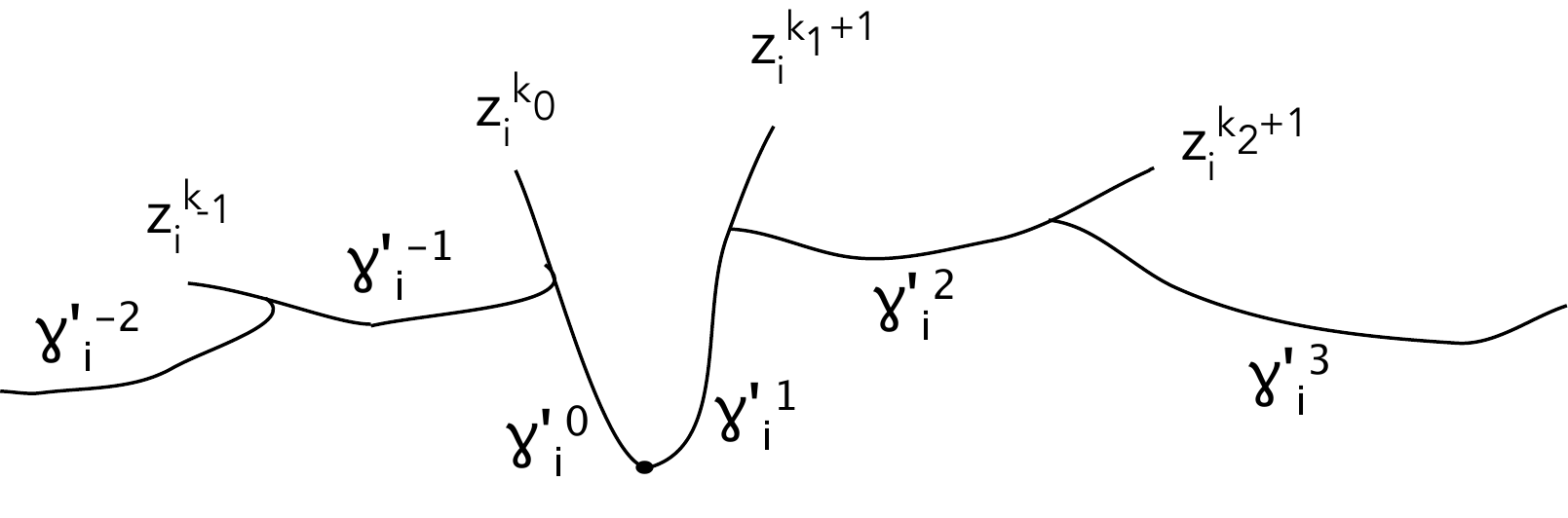}}}

\bigskip

These segments are free and have disjoint interiors. Moreover, if $ l'>l$, there exists $k>0$ such that $f^{k} (\gamma'{}_i^l)\cap \gamma'{}_i^{l'} \not=\emptyset$ because $\gamma'{}^l_i$ contains $z_i^{k_l+1}$ if $l>0$, and contains $z_i^{k_l}$ if $l\leq 0$. Consequently, by Proposition \ref{prop:guillouleroux}, for every $k>0$, one has 
 $f^{k} (\gamma'{}_i^{l'})\cap \gamma'{}_i^{l} \not=\emptyset.$ Note also that
\begin{itemize}

\item $\lim_{l\to -\infty} \gamma'{}_i^l =\alpha_i$;

\item $\lim_{l\to +\infty} \gamma'{}_i^l =\omega_i$.

\end{itemize}
In particular, there exists $L$ such that:
\begin{itemize}
\item $\gamma'_i{}^l\cap \gamma'_{i}{}^{l'}=\emptyset$ if $l<-L$ and $l'> L$;
\item  $\gamma'_i{}^l\cap \gamma'_{i'}{}^{l'}=\emptyset$  if $i\not=i'$, $\vert l\vert >L$ and $\vert l'\vert >L$. 
\end{itemize}

\subsection{Construction of an adapted brick decomposition} \label{ssection:adapted}We can find a free brick decomposition ${\mathcal D}'=(B',E',V')$ of $f$ defined on $\D\setminus\mathrm{fix}(f)$ and a family of bricks $(b'_i{}^m)_{m\in\Z\setminus \{0\}, \,i\in \Z/p\Z}$ in $B'$ such that $b'{}_i^m$ contains $\gamma'_i{}^{m+L}$ if $m>0$ and contains $\gamma'_i{}^{m-L}$ if $m<0$. To construct this decomposition we begin to enlarge the paths $\gamma'{}_i^l$, $\vert l\vert >L$, to construct the bricks $b'_i{}^m$ and then we decompose the complement of the union of these bricks in free bricks. Then we consider a maximal free disk decomposition ${\mathcal D}=(B,E,V)$ which is a sub-decomposition of ${\mathcal D}'$ and we denote $b_i^m$ the brick that contains $b'_i{}^m$. Note that for every $i\in\Z/p\Z$,  and for every $m$, $m'$ in $\Z\setminus\{0\}$ such that $m<m'$, there exists $k>0$ such that  $f^k(b_i^m)\cap b_i^{m'}\not=\emptyset$. In particular we have $b_i^m<b_i^{m'}$. Note also that for every $m>0$, the bricks $b_i^m$ and $b_i^{m+1}$ are adjacent. Consequently every brick $b_i^m$, $m>1$, is regular and $b_i^1$ is positively regular. Shifting the indexation if necessary (and so starting the sequence with $b_i^2$ instead of $b_i^1$) one can suppose that every $b_i^m$, $m>0$, is regular. Similarly, one can suppose that every  every $b_i^m$, $m<0$.

\bigskip
\bigskip

\centerline{{\includegraphics[scale=0.5]{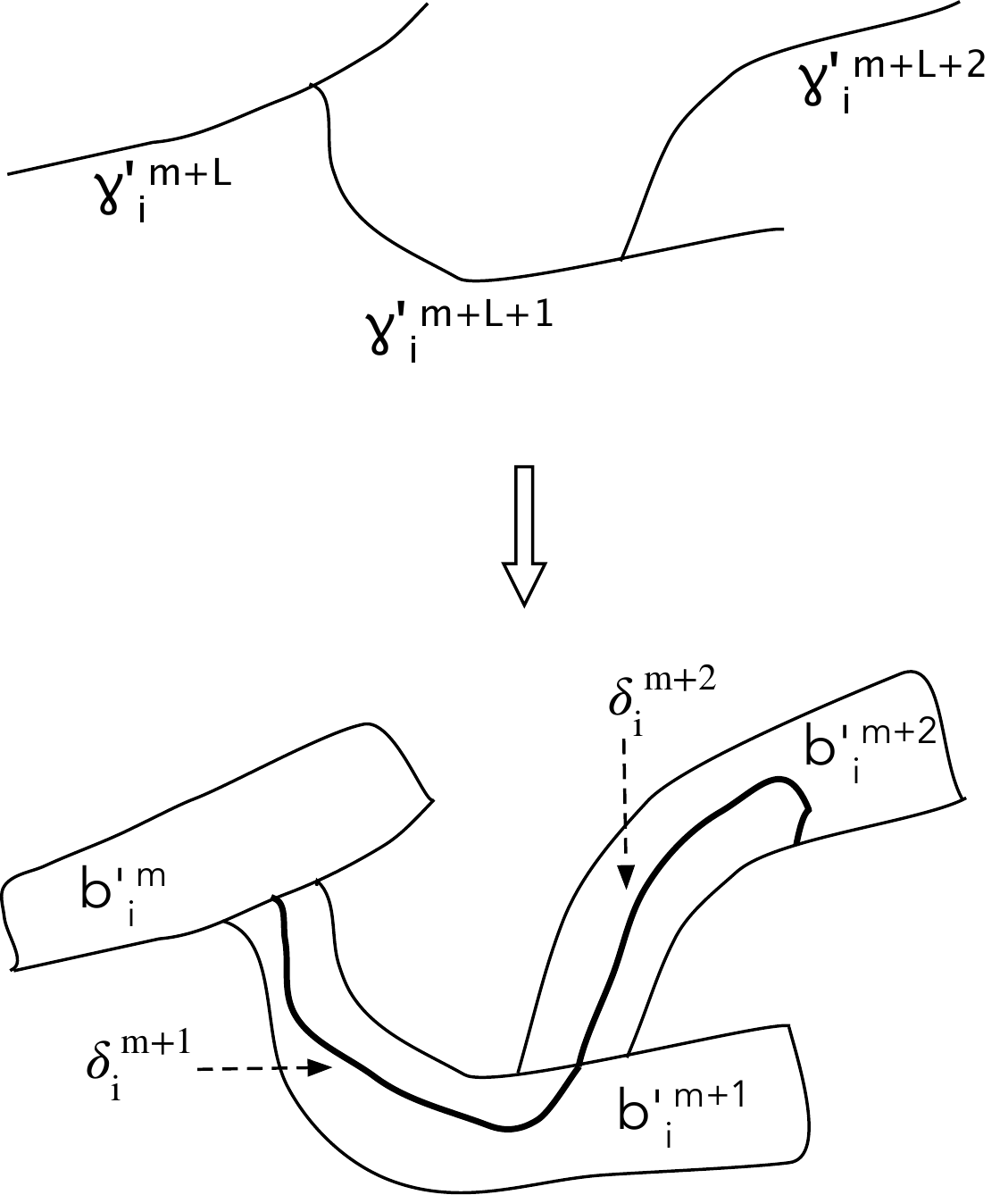}}}

\bigskip
\bigskip

To conclude, note that we can find a family of segments $(\delta_i^m)_{m\in\Z\setminus \{0\}, \,i\in \Z/p\Z}$ satisfying:
\begin{itemize}
\item $\partial \delta_i^m\subset \partial b'{}_i^m$;
\item $\mathrm{int}(\delta_i^m)\subset \mathrm{int}(b'{}_i^m)$;
\item if $m>0$, one obtains a half line $\lambda_i^{m}$ by concatenation of the $\delta_i^{m'}$, $m'\geq m$,  such that
\begin{itemize}
\item $\partial \lambda_i^m\subset \partial \left(\bigcup_{m'\geq m}b_i^{m'}\right)$;
\item $\mathrm{int}(\lambda_i^m)\subset \mathrm{int}\left(\bigcup_{m'\geq m}b_i^{m'}\right)$;
\item $\lim_{t\to+\infty}\lambda_i^m(t)=\omega_i$ if $\lambda_i^m$ is parametrized on $[0,+\infty)$
\end{itemize}
\item if $m<0$, then one obtains a half line $\lambda_i^{m}$ by concatenation of the $\delta_i^{m'}$, $m'\leq m$,  such that
\begin{itemize}
\item $\partial \lambda_i^m\subset \partial \left(\bigcup_{m'\leq m}b_i^{m'}\right)$;
\item $\mathrm{int}(\lambda_i^m)\subset \mathrm{int}\left(\bigcup_{m'\leq m}b_i^{m'}\right)$;
\item $\lim_{t\to-\infty}\lambda_i^m(t)=\alpha_i$ if $\lambda_i^m$ is parametrized on $(-\infty, 0]$
\end{itemize}
\end{itemize}

\section{Proof of Theorem  \ref{th:principal} }

 We consider the brick decomposition ${\mathcal D}=(B,E,V)$ built in the previous section and keep the notations introduced there. In this section, when we will talk about a component of a set $X\in{\mathcal P}(B)$, we will mean either a connected component of $X$ or the empty set. We will write $\Z^*=\Z\setminus\{0\}$.
 
\begin{lemma} \label{lemma:ends} Let $(R,A)$ be a cut of $\leq$. For every  $i\in\Z/p\Z$, there is a cut $(Z_i^-,Z_i^+)$ of $\Z^*$, a component $R_i$ of $R$ and a component $A_i$ of $A$ such that:
\begin{itemize}
\item $R_i$ contains  every $b_i^m$, $m\in Z_i^-$, and is empty if $Z_i^-$ is empty; 
\item $A_i$ contains  every $b_i^m$, $m\in Z_i^+$, and is empty if $Z_i^+$ is empty; 
\item  $R_i$ and $A_i$ are adjacent if $Z_i^-$ and $Z_i^+$ are not empty.
\end{itemize}
\end{lemma}

\begin{proof} Let $Z_i^-$ be the set of $m\in\Z^*$ such that $b_i^m\in R$ and $Z_i^+$ the set of $m\in\Z^*$ such that $b_i^m\in A$. The couple $(Z_i^-,Z_i^+)$ is a cut of  $\Z^*$ because $b_i^m\leq b_i^{m'}$ if $m\leq m'$. Every brick $b_i^m$, $m\in\Z^*$, is regular. So, for every $m\in Z_i^-$, the component of $R$ that contains $b_i^m$ is itself a repeller and contains every $b_i^{m'}$, $m'\leq m$. Similarly,  if $m\in Z_i^+$, the  component $A'$ of $A$ that contains $b_i^m$ is an attractor that contains every $b_i^{m'}$, $m'\geq m$. So, if $Z_i^-\not=\emptyset$, the bricks $b_i^m$, $m\in Z_i^-$, belong to the same connected component $R_i$ of $R$. Similarly, if $Z_i^+\not=\emptyset$, the bricks $b_i^m$, $m\in Z_i^+$, belong to the same connected component $A_i$ of $A$.   Defining $R_i=\emptyset$  if $Z_i^-=\emptyset$ and  $A_i=\emptyset$  if $Z_i^+=\emptyset$, we get the two first assertions. To get the third one,  use Proposition \ref{proposition:components}  and the fact that if $m<m'$, there exists $k>0$ such that $f^k(b_i^m)\cap b_i^{m'}\not=\emptyset$. \end{proof}

Say that $X\in{\mathcal P}(B)$ {\it contains} $\alpha_i$ if there exists $m<0$ such that $X$ contains every $b_i^{m'}$, $m'\leq m$ and that  $X$ {\it contains} $\omega_i$ if there exists $m>0$ such that $X$ contains every $b_i^{m'}$, $m'\geq m$.  

\begin{lemma} \label{lemma:transverse ends} Suppose that $\eta$, $\zeta$, $\eta'$, $\zeta'$, are four different points of
$$\{\alpha_i\, \vert\, i\in\Z/p\Z\}\cup \{\omega_i\, \vert\, i\in\Z/p\Z\},$$
such that $\eta'\in(\eta,\zeta)$ and $\zeta'\in(\zeta,\eta)$. If $X\in{\mathcal P}(B)$ is connected and contains $\eta$ and $\zeta$ and if $X'\in{\mathcal P}(B)$ is connected and contains $\eta'$ and $\zeta'$, then $X$ and $X'$ have at least one common brick.
\end{lemma}

\begin{proof}: Recall that we have defined half lines $\lambda_i^m$, $i\in\Z/p\Z$, $m\in\Z^*$, at the end of Section \ref{ssection:adapted}. The set $\mathrm{int}(X)$ is connected and contains two such half lines, one converging to $\eta$, the other one converging to $\zeta$. So, there is a line $\lambda\subset \mathrm{int}(X)$ that joins $\eta$ and $\zeta$. Similarly, there is a line $\lambda'\subset \mathrm{int}(X')$ that joins $\eta'$ and $\zeta'$. Consequently, one has $\lambda\cap \lambda'\not=\emptyset$, which implies that $X$ and $X'$ intersect in $B$.
\end{proof}

\begin{lemma} \label{lemma:alternative} If $(R,A)$ is a cut of $\leq$, then exactly one of the following situations holds:
\begin{enumerate}
\item there is a  component $R'$ of $R$ that contains  every $\alpha_i$, $i\in\Z/p\Z$;
\item there is a  component $A'$ of $A$ that contains  every $\omega_i$, $i\in\Z/p\Z$.
\end{enumerate}
Moreover, if {\em (1)} holds then, every component of $A$ contains at most one $\omega_i$, $i\in\Z/p\Z$, and if {\em (2)} holds, every component of $R$ contains at most one $\alpha_i$, $i\in\Z/p\Z$. 
\end{lemma}

\begin{proof}   By Lemma \ref{lemma:ends}, one knows that for every $i\in\Z/p\Z$, the set $R_i\cup A_i$ is connected and contains $\alpha_i$ and $\omega_i$. Suppose first that there exists $i\in\Z/p\Z$ such that $R_i=R_{i+1}$. Then $R_i$ contains $\alpha_i$ and $\alpha_{i+1}$, while $R_{i-1}\cup A_{i-1}$ contains $\alpha_{i-1}$ and $\omega_{i-1}$. By Lemma \ref{lemma:transverse ends}, it implies that $R_i$ and $R_{i-1}\cup A_{i-1}$ intersect in $B$, which means that $R_{i-1}=R_i$.  By induction one deduces that all $R_i$ are equal and so (1) holds.  Similarly, if $A_i=A_{i+1}$, then $A_{i+1}$ and $R_{i+2}\cup A_{i+2}$ intersect in $B$ and so $A_{i+1}=A_{i+2}$. One deduces that (2) holds. Applying Lemma  \ref{lemma:transverse ends} to $R_i\cup A_i$ and $R_{i+1}\cup A_{i+1}$, one deduces that these two sets intersect in $B$ and so, either $R_i=R_{i+1}$ or $A_i=A_{i+1}$. Consequently (1) or (2) holds. 

Suppose now that (1) holds. Applying again Lemma  \ref{lemma:transverse ends},  and the fact that $A\cap R'=\emptyset$ (in $B$). One deduces that every component of $A$ contains at most one $\omega_i$. Similarly, if (2) holds, every component of $R$ contains at most one $\alpha_i$. In particular (1) and (2) cannot occur simultaneously. \end{proof}

\begin{proof}[Proof of Theorem \ref{th:principal} ]
As explained in section \ref{section:order} there exists a total order $\preceq$ on $B$, which is weaker than $\leq$, meaning that for every $b$, $b'$ in $B$, we have:
$$b\leq b'\Rightarrow b\preceq b'.$$ 
In particular every cut of $\preceq$ is a cut of $\leq$. Denote ${\mathcal C} $ the set of cuts $\Gamma=(R_{\Gamma}, A_{\Gamma})$ of $\preceq$ and recall that $\preceq$ has a natural extension on ${\mathcal C}\sqcup B$.  
By Lemma \ref{lemma:alternative}, we have a partition ${ \mathcal C}= {\mathcal C}_- \cup{\mathcal C}_+$, where:
\begin{itemize}
\item $\Gamma$ belongs to ${\mathcal C}_-$ if there is a  component  of $A_{\Gamma}$ containing  every $\omega_i$;
\item $\Gamma$ belongs to ${\mathcal C}_+$ if there is a  component  of $R_{\Gamma}$ containing every $\alpha_i$. 
\end{itemize} 
It is clear that $({\mathcal C}_- ,{\mathcal C}_+)$\ is a cut of ${ \mathcal C}$. Note also that ${\mathcal C}_-$ and ${\mathcal C}_+$ are not  empty because $(\emptyset,B)\in {\mathcal C}_- $ and $(B,\emptyset)\in {\mathcal C}_+$.

\begin{lemma} \label{lemma:extrema} The cut $$ {\Gamma}_-^{\max} =\left(\bigcup_{\Gamma\in  {\mathcal C}_-} R_{\Gamma}, \bigcap_{\Gamma\in  {\mathcal C}_-} A_{\Gamma}\right)$$ is the greatest element of ${\mathcal C}_-$ and the cut $$ {\Gamma}_+^{\min} =\left(\bigcap_{\Gamma\in  {\mathcal C}_+} R_{\Gamma}, \bigcup_{\Gamma\in  {\mathcal C}_+} A_{\Gamma}\right)$$ the smallest element of ${\mathcal C}_+$. 
\end{lemma}

\begin{proof}

It is sufficient to prove that $ {\Gamma}_-^{\max} $ belongs to  ${\mathcal C}_-$,  the same proof will give us that $ {\Gamma}_+^{\min} $ belongs to  ${\mathcal C}_+$. We argue by contradiction and suppose that $ {\Gamma}_-^{\max} $ belongs to ${\mathcal C}_+$. There is a component $R'$ of  $R_{{\Gamma}_-^{\max}}$ that contains every $\alpha_i$ and so for every $i\in\Z/p\Z$, there exists $m_i<0$ such that $b_i^{m_i}\in R'$. This implies that there exists a finite connected subset $X\subset R_{{\Gamma}_-^{\max}}$ that contains every $b_i^{m_i}$, $i\in\Z/p\Z$. Using the equality $R_{{\Gamma}_-^{\max}}=\bigcup_{\Gamma\in  {\mathcal C}_-} R_{\Gamma}$ and the fact that ${ \mathcal C}$ is totally ordered, we deduce that there exists $\Gamma\in   {\mathcal C}_-$ such that $X\subset R_{\Gamma}$\footnote{It is here that we need to work with a total order} . This implies that the connected component of $R_{\Gamma}$ that contains $X$ is regular and contains all the $\alpha_i$, $i\in\Z/p\Z$, in contradiction with the definition of  ${ \mathcal C}_-$.\end{proof}

The set $R_{{\Gamma}_+^{\min}}$ is larger than $R_{{\Gamma}_-^{\max}}$, but $R_{{\Gamma}_+^{\min}}\setminus R_{{\Gamma}_-^{\max}}$ is reduced to a brick $b$, because there is no cut between ${\Gamma}_-^{\max}$ and  ${\Gamma}_+^{\min}$. So, we have $$ R_{{\Gamma}_+^{\min}} =R_{{\Gamma}_-^{\max}}\cup\{ b\}, \enskip   A_{{\Gamma}_-^{\max}} =A_{{\Gamma}_+^{\min}}\cup\{ b\}.$$ In particular, it holds that 
 $${\mathcal C}_-=\{\Gamma\in {\mathcal C}\,\vert\, \Gamma\prec b\},\enskip {\mathcal C}_+=\{\Gamma\in {\mathcal C}\,\vert\, b\prec \Gamma\}.$$ In other terms, a cut $(R, A)\in{ \mathcal C}$ satisfies the condition (1) of Lemma  \ref{lemma:alternative} if  $b\in R$ and  the condition (2) if $b\in A$.
 
Let us explain now where is the contradiction.
There exists a component $A'$ of $A_{{\Gamma}_-^{\max}}$ that contains every $\omega_i$, $i\in\Z/p\Z$, which means that for every $i\in\Z/p\Z$, there exists $m_i>0$ such that $b_i^{m_i}\in A'$. The set $A'$ being a connected attractor, it is the same for $\varphi_+(A')$. Moreover $\varphi_+(A')$ contains $b_i^{m_i+1}$ because $b_i^{m_i+1} \in \bigcup_{n>0} \varphi_+^n(\{b_i^{k_i}\})$ and $ \bigcup_{n>0} \varphi_+^n(\{b_i^{m_i}\})\subset \varphi_+(A')$. So $\varphi_+(A')$ is a connected attractor that contains every $\omega_i$, $i\in\Z$. The brick $b$ belongs to the attractor  $A_{{\Gamma}_-^{\max}}$ but is free. So, it holds that 
$$\varphi_+(\{b\} ) \subset  \left(A_{{\Gamma}_-^{\max}}\setminus\{b\}\right) =A_{{\Gamma}_+^{\min}}$$and so 
$$\varphi_+(A')\subset \varphi(A_{{\Gamma}_-^{\max}})= \varphi_+(A_{{\Gamma}_+^{\min}} )\cup\varphi_+(\{b\} ) \subset A_{{\Gamma}_+^{\min}}.$$ This contradicts the fact that no component of $A_{{\Gamma}_+^{\min}}$ contains all the $\omega_i$.
\end{proof}

\section{Further comments about the proof }

\subsection{Transverse foliations} As recalled in Section \ref{section:recurrent homeomorphisms}, the Brouwer Translation Theorem asserts that if $f$ is a Brouwer homeomorphism of $\D$, then $\D$ can be covered with Brouwer lines $\lambda$, meaning oriented such that $f(\overline{L(\lambda)})\subset L(\lambda)$
or equivalently such that
$f^{-1}(\overline {R(\lambda)})\subset R(\lambda)$. Its foliated version (see \cite{LeC1}) says the following: there exists a non singular $C^0$ foliation $\mathcal F$ on $\D$ such that every leaf is a Brouwer line. 

Let $\gamma:I\to \D$ be a path defined on a non trivial real interval. Say that $\gamma$ is {\it positively transverse} to $\mathcal F$ if it locally crosses leaves from the right to the left. Such a path must be injective because $\mathcal F$ is non singular. In fact, as explained in  \cite{HaeR}, the space of leaves furnished with the quotient topology is an oriented one dimensional
manifold, which is Hausdorff if and only if the foliation is trivial, meaning conjugate to the  foliation of $\C$ by verticals. A positively transverse path projects onto a simple path in the manifold $\mathcal F$, compatible with the natural orientation. Note that there is a natural order $\leq$ on $\mathcal F$ defined as follows: $\lambda<\lambda'$ if there exists a transverse path $\gamma:[0,1]\to \D$ such that $\gamma(0)\in\lambda$ and $\gamma(1)\in\lambda'$.  Let us state now an easy but important result proved in  \cite{LeC1}: for every $z\in \D$, there exists a positively transverse path that joins $z$ to $f(z)$. Noting $\phi_z$ the leaf containing a point $z\in\D$, this result tells that the function $z\mapsto \phi_z$ is a {\it Lyapunov function} of $f$ with values in $\mathcal F$: one has $\phi_{f(z)}<\phi_z$ for every $z\in \D$.  

Let us explain now how to use the foliated version of Brouwer Translation Theorem to prove that a homeomorphism satisfying the hypothesis of Theorem \ref{th:principal} has at least one fixed point. Here again, it will be a proof by contradiction. The simplest way to construct a transverse foliation is to start with a maximal free brick decomposition $\mathcal D = (B, E,V)$, to extend the order $\leq$ naturally induced on $B$ by a total order $\preceq$ and to consider the set ${\mathcal L}$ of lines $\lambda\subset \Sigma(\mathcal D)$, where $\lambda\in {\mathcal L}$ if and only if there exists a cut $c=(R_c, A_c)$ of $\preceq$ such that $\lambda$ is a connected component of $\partial R_c=\partial A_c$ (see \cite{LeC3}). The orientation of $\Sigma(D)$ defines a natural orientation on each $\lambda\in {\mathcal L}$, and then $\lambda\in {\mathcal L}$ is a Brouwer line. There is a natural topology and a natural order on ${\mathcal L}$ (partial but locally total ) that makes  ${\mathcal L}$ similar to a lamination (every $\lambda$ has a neighborhood that is homeomorphic to a totally disconnected compact subset of a real interval). By a process of desingularization, one can ``blow up'' the space  ${\mathcal L}$ and transform it into a real lamination of $\D$ by Brouwer lines, homeomorphic to $\mathcal L$. To construct $\mathcal F$, it remains to fill the complement of this lamination. 

The hypothesis of Theorem \ref{th:principal} tells us that, for every $i\in\Z/p\Z$, there exists a line $\Gamma'_i$ that is transverse to $\mathcal F$ and that passes through every point $f^k(z_i)$, $k\in\Z$. Of course this line accumulates onto $\alpha_i$ and $\omega_i$ but can have other accumulation points on $\S$.  If we start with the adapted brick decomposition defined in Section {\ref{section:adapted}, we get an extra condition: there exists a line $\Gamma_i$ that is transverse to $\mathcal F$,whose projection in the space of leaves is the same as the projection of $\Gamma'_i$ and that accumulates only onto $\alpha_i$ and $\omega_i$.

It is easy now to find a contradiction. The important fact being that $\Gamma_i\cap\Gamma_{i+1}\not=\emptyset$, for every $i\in \Z$.
A first argument is to consider the fonction $\nu_i$ defined on the complement of $\Gamma_i$, equal to $0$ on $R(\Gamma_i)$ and to $1$ on $L(\Gamma_i)$, then to consider the fonction $\nu=\sum_{i\in\Z/p\Z}\nu_i$. In a neighborhood of $\S$, when it is defined, the function $\nu$ takes the values $0$ or $1$. The fact that $\Gamma_i\cap\Gamma_{i+1}\not=\emptyset$ implies that $\nu$ takes at least $3$ values on its domain of definition, and so either it holds that $\max \nu>1$ or it holds that $\min \nu<0$. Suppose that we are in the first case and consider a connected component $U$ of the domain of definition of $\nu$.   
It is relatively compact. It is easy to show that $U$ is the interior of a surface with boundary, whose boundary is transverse to $\mathcal F$, the leaves leaving $U$. This is incompatible with the fact that the leaves are lines of $\D$ (one can also say that there is no simple loop transverse to $\mathcal F$). In the case where $\min \nu<0$ we have a similar situation with leaves entering into the domain. 

Let us give another argument. For every $i\in\Z/p\Z$ and every $z\in \Gamma_i$ denote $\Gamma^-_i(z)$ and $\Gamma^+_i(z)$ the half lines contained in $\Gamma_i$ arriving at $z$ and starting at $z$ respectively. Let us consider the first point $z\in \Gamma_i\cap\Gamma_{i+1}$ where $\Gamma_i$ meets $\Gamma_{i+1}$. The union of $\Gamma^-_i(z)$ and $\Gamma^-_{i+1}(z)$ is a line. Moreover, the component of its complement that contains the half-lines $\Gamma^+_{i-1}(z')$, $z'$ close to $\omega_i$, is included in $R(\phi_z)$. Consequently
$R(\phi_z)$ contains the whole line $\Gamma_{i-1}$. One deduces that $R(\phi_z)$ contains $\Gamma^+_{i-2}(z')$ if $z'$ is close to $\omega_{i-2}$ and so it contains $\Gamma_{i-2}$. By iteration of this process we prove that $R(\phi_z)$ contains all $\Gamma_{i'}$, $i'\in\Z/p\Z$, which of course is absurd.

The original idea to get a proof of Handel's theorem via Brouwer Theory was to use the argument above. Unfortunately, Lemma \ref{lemma:new} was  missing, making impossible the use of the lines $\Gamma_i$, $i\in\Z/p\Z$, as above (the same argument does not work with the lines $\Gamma'_i$). What is proved and used in \cite{LeC2} but also in \cite{X1} and \cite{X2} is the following:

\begin{lemma} \label{lemma:old}There exists a sequence $(\gamma_i^k)_{k\in\Z}$ of translation arcs such that:

\begin{enumerate}

\item $\gamma_i^k$ joins $f^k(z_i)$ to $f^{k+1}(z_i)$;

\item $f(\gamma_i^k)\cap\gamma_i^{k'}=\emptyset$ if $k'<k$;

\item $\lim_{k\to -\infty} \gamma_i^k =\alpha_i$;

\item $\lim_{k\to +\infty} \gamma_i^k =\omega_i$.
\end{enumerate}

\end {lemma}

Surprisingly, the proof of Lemma \ref{lemma:new} is much simpler than the proof of Lemma \ref{lemma:old}. The second condition of  Lemma \ref{lemma:new}, the fact that $\gamma_i^k$ and $\gamma_i^{k+1}$ coincide in a neighborhood of $z_i^k$, is fundamental in the construction of the adapted brick decomposition, but as it can be seen in the proof of Lemma \ref{lemma:new}, is very easy to get.

It is a natural question whether it is possible to construct a similar foliation in case $f$ is not recurrent and get a complete proof of Theorem \ref{th:principal}. The answer is probably yes. Nevertheless  it is not clear that it will pertinent to try to do so. Let us give a first reason. To construct a transverse foliation for a Brouwer homeomorphism, it is necessary to prove first that there is no singular brick. The characterization that is given in Lemma
\ref{lemma:singular brick} lead us to believe that singular bricks do not exist in the more general case of non recurrent map. But the proof will not be too easy and will certainly need to use topological arguments that cannot be expressed in terms of brick decompositions, like the majority of the arguments used here. In fact, singular bricks do not really cause problems (we will see very soon how to avoid them if necessary). A second reason not to use foliations is that the desingularisation and the filling processes necessary to get a foliation is quite complicated, long and not easy to write. Moreover, Brouwer lines should be replaced by more complex manifolds (unions of lines)  like what is done in the thesis of Tran Ngoc Diep \cite{T}, where a Brouwer type foliation is constructed for some orientation reversing plane homeomorphisms.

\subsection{As a conclusion}

We will conclude the article with some comments, that should permit us to understand why working with brick decompositions is not very different from working with transverse foliations. Let us first display a special case, where the proof of Theorem \ref{th:principal} is particularly simple. Let ${\mathcal D}=(B,E,V)$ be the adapted brick decomposition built in Section  \ref{section:adapted}. Suppose that ${\mathcal D}$  satisfies the following condition: 
for every bricks $b$, $b'$ such that $b\leq b'$, there exists a sequence
$(b_j)_{0\leq j\leq n}$ in $B$ such that:
\begin{itemize}
\item $b_0=b$ and $b_n=b'$;
\item$b_j<b_{j+1}$ for every
$j\in\{0,\dots,n-1\}$;
\item $b_j$ and $b_{j+1}$ are adjacent for every $j\in\{0,\dots,n-1\}$.
\end{itemize}
In that case, one easily constructs, for every $i\in\Z/p\Z$, a sequence $(\hat b_i^m)_{m\in\Z}$ and an integer $M_i\geq 0$ such that 
\begin{itemize}
\item $\hat b_i^m=b_i^m$, for every $m<0$;
\item $\hat b_i^{m}=b_i^{m-M_i}$, for every $m>M_i$;
\item$\hat b_i^m<\hat b_i^{m+1}$ for every
$m\in\Z$;
\item $\hat b_i^m$ and $\hat b_i^{m+1}$ are adjacent for every $m\in\Z$.
\end{itemize}
Note that the $\hat b_i^m$, $i\in\Z/p\Z$, $m\in\Z$, are all regular, note also that the set $\{\hat b_i^m\,,\, m\in\Z\}$ is connected and contains $\alpha_i$ and $\omega_i$. So, for every $i\in\Z/p\Z$, there exists $m_i\in\Z$ and $m'_i\in\Z$ such that $\hat b_i^{m_i}= \hat b_{i+1}^{m'_i}$.  Consider the cut $( (\hat b_i^{m_i})_{\leq}, (\hat b_i^{m_i})_{>} )$. Both sets $(\hat b_i^{m_i})_{\leq}$  and $(\hat b_i^{m_i})_{>} $ are connected  because $\hat b_i^{m_i}$ is regular. The first set contains $\alpha_{i}$ and $\alpha_{i+1}$ and the second one contains $\omega_i$ and $\omega_{i+1}$. By Lemma  \ref{lemma:alternative} one deduces that $(\hat b_i^{m_i})_{\leq}$   contains every $\omega_{i'}$, $i'\in\Z/pZ$, and  $(\hat b_i^{m_i})_{>} $ contains every $\alpha_{i'}$, $i'\in\Z/pZ$. This contradicts Lemma \ref{lemma:alternative}.

The last proposition, tells us that this ``connectedness'' property remains valid if one adds the cuts to the bricks. What we get is analogous to the fact that, for a Brouwer homeomorphism furnished with a transverse foliation, every point can be joined to its image by a transverse path. 

\begin{proposition} \label{prop:connectedness} We suppose that $f$ is a non recurrent orientation preserving homeomorphism of the $2$-sphere $\S^2$ and that ${\mathcal
D}=(V,E,B)$ is a maximal free brick decomposition of $f_{\vert \S^2\setminus\mathrm{fix}(f)}$. We denote $\leq$ the induced order on $B$ and consider a total order $\preceq$ weaker than $\leq$.  Then, for every regular bricks $b$, $b'$ such that $b\leq b'$, there exists a sequence
$(b_j)_{0\leq j\leq n}$ of regular bricks, satisfying $b_0=b$ and $b_n=b'$, and such that for every $j\in\{0,\dots, n-1\}$, it holds that:
\begin{itemize}
\item $b_j<b_{j+1}$;
\item $b_{j+1}$ is adjacent to the component of $(b_j)_{\preceq}$ that contains $b_{j}$; 
\item $b_j$ is adjacent to the component of $(b_j)_{\succ}$ that contains $b_{j+1}$.
\end{itemize}\end{proposition}

\begin{proof} The brick $b'$ being regular, the set $b'_{\leq}$ is connected, and by hypothesis, it contains $b$. So, there exists a sequence
$(\hat b_k)_{0\leq k\leq m}$ in $b'_{\leq}$, such that:

\begin{itemize}
\item $\hat b_0=b$ and $\hat b_m=b'$; 
\item  $\hat b_k$ and $\hat b_{k+1}$ are adjacent for every every $k\in\{0,\dots, m-1\}$. 
\end{itemize}
Let us explain first why the bricks can be supposed to be positively regular. Indeed suppose that there exists $k\in\{0,\dots, m\}$ such that $\hat b_k$ is positively singular. By hypothesis, $k\not\in\{0,m\}$ and the bricks $\hat b_{k-1}$ and $\hat b_{k+1}$ are adjacent to $\hat b_k$. We have proved in Lemma \ref{lemma:singular brick}  than $\partial \hat b_k$ is reduced to a line of $\D\setminus \mathrm{fix}(f)$. This implies that there exists a sequence $(\check b_l)_{0\leq l\leq q}$ in $B$ such that:
\begin{itemize}
\item $\check b_0=\hat b_{k-1}$ and $\check b_q=\hat b_{k+1}$;
\item $\check b_l$ is adjacent to $\hat b_k$ for every every $l\in\{0,\dots,q\}$;
\item $\check b_l$ and $\check b_{l+1}$ are adjacent for every $l\in\{0,\dots, q-1\}$. 
\end{itemize}
Note that the bricks $\check b_l$, $l\in\{0,\dots, q\}$, are all positively regular and included in $b'_{\leq}$ because we have $\check b_l\leq \hat b_k  \leq b'$. So, this process permit to avoid positively singular bricks by extending our original sequence of bricks. Then we can shorten our sequence in such a way that a brick appears at most once. 

Let us consider now the increasing sequence $(k_j)_{0\leq j\leq n}$ uniquely defined by the following properties:
 \begin{itemize}
 \item $k_0=0$ and $k_n=m$;
\item  $\hat b_{k_j}\succ \hat b_{k_{j-1}}$ if $0<j\leq n$;
\item  $\hat b_{k}\preceq \hat b_{k_{j-1}}$ if $0<j\leq n$ and $0\leq k<k_j$.
\end{itemize}
Setting $b_j=\hat b_{k_j}$, we will prove that the sequence $(b_j)_{0\leq j\leq n}$ satisfies the conclusions of Proposition \ref{prop:connectedness}.

We denote ${\mathcal C}$ the set of cuts of $\preceq$ and define  $c_i=( (b_i)_{\preceq}, (b_i)_{\succ})\in {\mathcal C}$, noting that 
$b_0\prec c_0\prec b_1\prec \dots\prec c_{n-1}\prec b_n$.

Note first that the bricks $b_j$, $0\leq j\leq n$, are all regular. It is true if $j\in\{0, n\}$ by hypothesis. Suppose now that $0<j<n$. The bricks $\hat b_{k_j-1}$ and $\hat b_{k_j}$ are comparable for $\leq$ because they are adjacent and we have $\hat b_{k_j-1}\preceq \hat b_{k_{j-1}} \prec \hat b_{k_j}$. So it holds that $\hat b_{k_j-1}<\hat b_{k_j}$. Consequently, $\hat b_{k_j}$ is regular, being positively regular by hypothesis. 

Fix $j\in\{0,\dots, n-1\}$. The brick $b_j$ being regular, the set $(b_j)_{>}$ is connected, and by hypothesis, it contains $b'$. So it is contained in the component $A_j$ of $(b_j)_{\succ}$ that contains $b'$. Note that $b_j$ is adjacent to $A_j$ because, being positively regular,  it is adjacent to $(b_j)_{>}$. Note $R_j$ the component of $(b_j)_{\preceq}$ that contains $b_j$.  It contains all $b_{k}$, $0\leq k<k_{j+1}$, because these bricks belong to $(b_j)_{\preceq}$ and their union is connected. Consequently $b_{j+1}=\hat b_{k_{j+1}}$ is adjacent to $R_j$. The fact that $b_{j+1}$ is positively regular implies that the component of $(b_j)_{\succ}$ that contains $b_{j+1}$ is regular and so is an attractor. One deduces that $(b_{j+1})_{\geq}$ is included in this component. By hypothesis $b'\geq b_{j+1}$ and so,  $b_{j+1}$ belongs to $A_j$.
\end{proof}




\begin{thebibliography}{Ha}




















\bibitem[Ba]{Ba}
{\sc J. Bavard~:}
 \newblock{Conjugacy invariants for Brouwer mapping classes,}
 \newblock{ \em Ergodic Theory Dynam. Systems} {\bf 37} (2017), no. 6, 1765--1814.




\bibitem[Br]{Br}
{\sc L. E. J. Brouwer~:}
\newblock Beweis des ebenen Translationssatzes,
\newblock {\em Math. Ann.}, {\bf 72} (1912), 37-54.


\bibitem[Bn]{Bn}
{\sc M. Brown~:}
\newblock A new proof of Brouwer's lemma on translation arcs,
\newblock {\em Houston J. Math.}, {\bf 10} (1984), 35-41.


\bibitem[Fa]{Fa}
{\sc A. Fathi~:}
\newblock An orbit closing proof of Brouwer's lemma on translation 
arcs,
\newblock {\em Enseign. Math.}, {\bf 33}
(1987), 315-322.


\bibitem[Fl]{Fl}
{\sc A. Floer~:}
\newblock Proof of the Arnold conjectures for surfaces and 
generalizations
to certain K\"ahler
manifolds,
\newblock {\em Duke Math. J.}, {\bf 51} (1986), 1-32.








\bibitem[Fr1]{Fr1}
{\sc J. Franks~:}
\newblock Generalizations of the Poincar\' e-Birkhoff theorem,
\newblock {\em Annals of Math.}, {\bf 128} (1988), 139-151.





\bibitem[Fr2]{Fr2}
{\sc J. Franks~:}
\newblock Rotation vectors and fixed points of area preserving surface diffeomorphisms,
\newblock {\em Trans. Amer. Math. Soc}, {\bf 348} (1996), 2637-2662.



\bibitem[Fr3]{Fr3}
{\sc J. Franks~:}
\newblock Area preserving homeomorphisms of open surfaces of genus 
zero,
\newblock {\em New-York J. Math.}, {\bf 2} (1996), 1-19.








\bibitem[G]{G}
{\sc L. Guillou~:}
\newblock Th\'eor\`eme de translation plane de Brouwer et
g\'en\'eralisations du
  th\'eor\`eme de Poincar\'e-Birkhoff,
\newblock {\em Topology}, {\bf 33} (1994), 331--351.



\bibitem[Ha]{Ha}
{\sc M. Handel~:}
\newblock A fixed point theorem for planar homeomorphisms,
\newblock {\em Topology}, {\bf 38} (1999), 235-264.


\bibitem[HaeR]{HaeR}
{\sc A. Haefliger, G. Reeb :} 
\newblock Vari\'et\'es (non s\'epar\'ees) \`a une dimension et structures feuillet\'ees du plan. 
\newblock \textit{Enseign. Math. }(2) {\bf 3} (1957), 107--125






\bibitem[LeC1]{LeC1}
{\sc  P. Le Calvez~:} 
\newblock Une version feuillet\'ee du th\'eor\`eme de translation de Brouwer. 
\newblock{\em Comment. Math. Helv.}{\bf  79} (2004), no. 2, 22--259. 



\bibitem[LeC2]{LeC2}
{\sc  P. Le Calvez~:} 
\newblock Une nouvelle preuve du th\'eor\`eme de point fixe de Handel. 
\newblock {\em Geom. Topol. }{\bf 10} (2006), 2299--2349. 

 
\bibitem[LeC3]{LeC3}
{\sc  P. Le Calvez~:} Identity isotopies on surfaces. Dynamique des diff\'eomorphismes conservatifs des surfaces. Un point de vue topologique. Paris: Soci\'et\'e Math\'ematique de France. Panoramas et Synth\`eses {\bf 21} (2006), 105--143 . 


\bibitem[LeR1]{LeR1}
{\sc  F. Le Roux~:} 
\newblock Hom\'eomorphismes de surfaces: th\'eor\`emes de la fleur de Leau-Fatou et de la vari\'et\'e stable, 
\newblock {\em Ast\'erisque } {\bf 292} (2004).



\bibitem[LeR2]{LeR2}
{\sc  F. Le Roux~:} 
\newblock An index for Brouwer homeomorphisms and homotopy Brouwer theory. 
\newblock {\em Ergodic Theory Dynam. Systems} {\bf 37} (2017), no. 2, 572--605. 


 

\bibitem[M]{M}
{\sc S. Matsumoto~:}
\newblock Arnold conjecture for surface homeomorphisms, Proceedings 
of the
French-Japanese
Conference ``Hyperspace Topologies and Applications '' (La 
Bussi\`ere, 1997),
\newblock {\em Topology. Appl.}, {\bf 104} (2000), 191-214.



\bibitem[Sa]{Sa}
{\sc A. Sauzet~:}
\newblock Application des d\'ecompositions libres \`a l'\'etude des
hom\'eomorphismes de surface,
\newblock {\it Th\`ese de l'Universit\'e Paris 13}, (2001).


\bibitem[Si]{Si}
{\sc J.-C. Sikorav~:}
\newblock Points fixes d'une application symplectique homologue \`a
l'identit\'e,
\newblock {\it J. Diff. Geom}, {\bf 22} (1985), 49-79.



\bibitem[T]{T}
{\sc N.-D. Tran~:}
\newblock 
Sur la dynamique des hom\'eomorphismes de surfaces qui renversent l'orientation,
\newblock {\it PhD Thesis, Universit\'e Paris13}, (2018).





\bibitem[X1]{X1}
{\sc J. Xavier~:}
\newblock{Cycles of links and fixed points for orientation preserving homeomorphisms of the open unit disk.}
\newblock{ \em Fund. Math.}, {\bf 219} (2012), no. 1, 59--96. 

\bibitem[X2]{X2}
{\sc J. Xavier~:}
\newblock Handel's fixed point theorem revisited, 
\newblock {\em Ergodic Theory Dynam. Systems}, {\bf 33} (2013), no. 5, 1584--1610. 






\end{thebibliography}
\end{document}